\documentclass[12pt,reqno]{amsart}

\title[Invariant Hilbert scheme resolution: the non-toric case]{Invariant Hilbert scheme resolution of Popov's $SL(2)$-varieties II: the non-toric case}
\author{Ayako Kubota}
\address{
Department of Mathematics, Graduate School of Fundamental Science and Engineering, 
Waseda University, 
3-4-1 Ohkubo, Shinjuku, Tokyo 169-8555, Japan}
\email{ayako.kubota.math@gmail.com}
\usepackage{amsthm}
\usepackage[reqno]{amsmath}
\usepackage{amssymb}
\usepackage{amsxtra}
\usepackage[mathscr]{eucal}
\usepackage{enumerate}
\usepackage[all]{xy}
\usepackage[initials,alphabetic,lite]{amsrefs}
\usepackage{youngtab}
\usepackage{graphics}
\usepackage{url}

\usepackage{times}
\usepackage{mathptmx}
\usepackage{newtxtext}
\usepackage{newtxmath}
\usepackage{rsfso}
\usepackage{tikz}

\theoremstyle{plain}

\newtheorem{theorem}[subsection]{Theorem}
\newtheorem{lemma}[subsection]{Lemma}
\newtheorem{proposition}[subsection]{Proposition}
\newtheorem{corollary}[subsection]{Corollary}

\newtheorem*{theorem*}{Theorem}
\newtheorem*{corollary*}{Corollary}
\newtheorem*{MainTheorem}{Main Theorem}
\theoremstyle{definition}
\newtheorem{definition}[subsection]{Definition}
\newtheorem{example}[subsection]{Example}

\theoremstyle{remark}

\newtheorem{remark}[subsubsection]{Remark}

\DeclareSymbolFont{cmletters}{OML}{cmm}{m}{it}
\DeclareSymbolFont{cmsymbols}{OMS}{cmsy}{m}{n}
\DeclareSymbolFont{cmlargesymbols}{OMX}{cmex}{m}{n}
\DeclareMathSymbol{\myjmath}{\mathord}{cmletters}{"7C}
\DeclareMathSymbol{\myamalg}{\mathbin}{cmsymbols}{"71}
\DeclareMathSymbol{\mycoprod}{\mathop}{cmlargesymbols}{"60}
\let\jmath\myjmath

\DeclareMathOperator{\Hilb}{Hilb}
\DeclareMathOperator{\Sym}{Sym}

\DeclareMathOperator{\Spec}{Spec}

\DeclareMathOperator{\pr}{pr}
\DeclareMathOperator{\id}{id}

\DeclareMathOperator{\Hom}{Hom}

\DeclareMathOperator{\diag}{diag}
\DeclareMathOperator{\Gr}{Gr}

\DeclareMathOperator{\Ker}{Ker}

\DeclareMathOperator{\Irr}{Irr}

\DeclareMathOperator{\Div}{div}
\DeclareMathOperator{\Car}{Car}
\DeclareMathOperator{\Rem}{Rem}

\def\git{/\!\!/}
\def\G{G_0 \times G_m}

\def\C{\mathbb C}
\def\E{E_{l,m}}
\def\P{\mathbb P}
\def\w{\omega}
\def\Z{\mathbb Z}

\makeatletter
\newcommand{\biggg}[1]{{\hbox{$\left#1\vbox to 20.5pt{}\right.\n@space$}}}
\newcommand{\Biggg}[1]{{\hbox{$\left#1\vbox to 23.5pt{}\right.\n@space$}}}
\newcommand{\bigggg}[1]{{\hbox{$\left#1\vbox to 26.5pt{}\right.\n@space$}}}
\newcommand{\Bigggg}[1]{{\hbox{$\left#1\vbox to 29.5pt{}\right.\n@space$}}}
\newcommand{\biggggg}[1]{{\hbox{$\left#1\vbox to 32.5pt{}\right.\n@space$}}}
\newcommand{\Biggggg}[1]{{\hbox{$\left#1\vbox to 35.5pt{}\right.\n@space$}}}
\newcommand{\bigggggg}[1]{{\hbox{$\left#1\vbox to 38.5pt{}\right.\n@space$}}}
\newcommand{\Bigggggg}[1]{{\hbox{$\left#1\vbox to 41.5pt{}\right.\n@space$}}}
\makeatother

\begin{document}

\baselineskip 17pt
\parskip 7pt

\maketitle
\vspace{-1cm}

\begin{abstract}
This article is a continuation of \cite{Ku}, which proves that if a $3$-dimensional affine normal quasihomogeneous $SL(2)$-variety $E$ is toric,   then it has an equivariant resolution of singularities given by an invariant Hilbert scheme $\mathcal H$. 
In this article, we consider the case where $E$ is non-toric and show that the Hilbert--Chow morphism $\gamma : \mathcal H \to E$ is a resolution of singularities 
and that $\mathcal H$ is isomorphic to 
the minimal resolution of a weighted blow-up of $E$. 
\end{abstract}


\section*{Introduction}
Let $\E$ be a $3$-dimensional affine normal quasihomogeneous $SL(2)$-variety of height $l$ and degree $m$, and write $l$ as an irreducible fraction $l=p/q$. 
Batyrev and Haddad \cite{BH} showed that 
$\E$ has a description as an affine categorical quotient of a hypersurface $H_{q-p}$ in $\C^5$ modulo an action of $\C^* \times \mu_m$. 
Also, they proved that an $SL(2)$-variety $\E$ admits an action of $\C^*$ and becomes a spherical $SL(2) \times \C^*$-variety with respect to the Borel subgroup $B \times \C^*$. 
Further, it is shown that there is an  
equivariant flip diagram 
\[
\xymatrix@R=5pt{
\E^- \ar@{.>}[rr] \ar[rd] & & \E^+ \ar[ld] \\
& \E &
},
\]
where $\E^{-}$ and $\E^{+}$ are different GIT quotients of $H_{q-p}$ corresponding to some non-trivial characters, 
and that the varieties $\E$, $\E^-$, and $\E^+$ are dominated by the weighted blow-up $\E'=Bl_O^{\w}(\E)$ of $\E$ with a weight $\w$ defined by the above-mentioned $\C^*$-action on $\E$. 
The weight $\w$ is trivial if and only if the $SL(2)$-variety $\E$ is toric, namely if $m=a(q-p)$ holds for some $a>0$ (see \cite{G, BH}). 

In our previous article  \cite{Ku}, we used the GIT quotient 
description of $\E$ due to Batyrev and Haddad to construct the invariant Hilbert scheme $\mathcal H=\Hilb^{\C^* \times \mu_m}_h(H_{q-p})$,  
where $h$ is the Hilbert function of the general fibers of the quotient morphism $H_{q-p} \to H_{q-p} \git (\C^* \times \mu_m)$,  
and considered the corresponding Hilbert--Chow morphism 
\[
\xymatrix{
\gamma : \mathcal H \ar[r] &  H_{q-p} \git (\C^* \times \mu_m) \cong \E
},
\]
which is an isomorphism over the dense open orbit $\mathfrak U \subset \E$. 
We treated the case where $\E$ is toric and showed that  
the main component $\mathcal H^{main}=\overline{\gamma^{-1}(\mathfrak U)}$ is isomorphic to the blow-up $\E'$ and that $\mathcal H$ coincides with $\mathcal H^{main}$ (\cite
[Corollary 5.2 and Theorem 6.1]{Ku}). 

The goal of this article is to prove the following result. 
\begin{MainTheorem}[Theorem \ref{main theorem} and Corollary \ref{duval}]
If $\E$ is non-toric, then:
\begin{itemize}
\item [\rm(i)] the main component $\mathcal H^{main}$ is isomorphic to the minimal resolution $\widetilde{\E'}$ of the weighted blow-up $\E'$; 
\item [\rm(ii)] the invariant Hilbert scheme $\mathcal H$ coincides with the main component $\mathcal H^{main}$. 
\end{itemize}
\end{MainTheorem}
The problem of deciding the main component $\mathcal H^{main}$ and showing the smoothness of $\mathcal H=\mathcal H^{main}$ is easier in the toric case than in the non-toric case, since we have the relation $m=a(q-p)$. 
 The non-toric case requires more intricate arguments, which is mainly because there is nothing to relate the height $l=p/q$ and the degree $m$ directly, but the essential idea for the proof is the same as in the toric case. 
 In the following, we outline our approach for the non-toric case. 
 First, as in the toric case, we show that the restriction $\gamma|_{\mathcal H^{main}}$ factors equivariantly through the weighted blow-up $\E'$: 
\[
\xymatrix{
\mathcal H^{main} \ar[rr]^{\; \; \; \; \psi} \ar[rrd]_{\gamma |_{\mathcal H^{main}}} &  & \E'  \ar[d]\\
& & \E
}
\]
If $\E$ is toric then $\psi$ is an isomorphism, while if $\E$ is non-toric then 
we see by an                                                                                                                                                                                                                                                                 easy observation that $\psi$ is not an isomorphism. 
On the other hand, according to \cite{BH}, the weighted blow-up  $\E'$ contains a family of cyclic quotient singularities $\C^2/ \mu_b$, and  
therefore the natural candidate for $\mathcal H^{main}$ is 
the minimal resolution  $\widetilde{\E'}$ of these quotient singularities, which is known to be described by the Hirzebruch--Jung continued fraction. 
So what we do next 
is to construct an equivariant morphism $\mathcal H^{main} \to \widetilde{\E'}$: 
we first realize $\widetilde{\E'}$ as a closed subscheme of a projective space over $\E$ and then use Becker's idea \cite[\S 4]{Bec} of embedding an invariant Hilbert scheme to products of Grassmannians to 
construct a morphism $\Psi$ from $\mathcal H$ to the projective space such that $\Psi(\mathcal H^{main}) \cong \widetilde{\E'}$. 
Finally, we show that $\Psi|_{\mathcal H^{main}}: \mathcal H^{main} \to \widetilde{\E'}$ is an isomorphism. 
By the Zariski's Main Theorem, it suffices to show that $\Psi|_{\mathcal H^{main}}$ is injective, and concerning that it is equivariant we are left to show the injectivity orbit-wise: 
we take a ``representative'' point from each orbit in $\widetilde{\E'}$ (e.g. we take a Borel-fixed point if the orbit is closed) and show that its fiber consists of one point, say $[Z] \in \mathcal H^{main}$. 
In showing the injectivity, the differences from the toric case are that the number of orbits in $\widetilde{\E'}$ depends on the pair $(l,m)$ and that the degrees of generators of the ideal $I_Z$ of $Z$ can not be expressed in terms of $p,\; q$, or $m$. 
What becomes a key here is the spherical geometry of $\widetilde{\E'}$, which enables us to give a uniform approach independent of the pair $(l,m)$. 
To be more precise,  
the number of orbits can be read off from the colored fan of $\widetilde{\E'}$. 
Also, the degrees of the generators of $I_Z$ are described by using ray generators of maximal cones contained in the fan of $\widetilde{\E'}$, and the relations among them come from recursive relations arising from the Hirzebruch--Jung continued fraction. 
This is why the calculation of generators of the ideal $I_Z$ 
involves intricate combinatorial arguments in contrast to the toric case. 

This article is organized as follows: 
we first summarize some general properties of invariant Hilbert schemes in \S \ref{s-generalhilb} and of spherical varieties in \S \ref{s-spherical}. 
Afterwards, we review Popov's classification of $3$-dimensional affine normal quasihomogeneous $SL(2)$-varieties  (Theorem \ref{popov}) and the GIT quotient description due to Batyrev and Haddad (Theorems \ref{batyrev} and \ref{batyrev2}). 
In \S \ref{s-flat}, we first review some facts from \cite{Ku}, and then describe the minimal resolution $\widetilde{\E'}$ in terms of its colored fan. 
In \S \ref{s-proof}, we realize $\widetilde{\E'}$ as a closed subscheme of a projective space over $\E$ by using the spherical geometry of $\widetilde{\E'}$ (Proposition \ref{phicl}).  
\S \ref{s-gen} is a preparation for later sections and is mainly devoted to the proof of Theorem \ref{generator2}, which requires some complicated combinatorial arguments. 
In \S \ref{s-second}, we construct the morphism $\Psi$ by using Theorem \ref{generator2}. 
In \S \ref{s-idealo}, we calculate ideals (Theorems \ref{idealO2} and \ref{idealO}), which will be shown to correspond to  ``representative'' points in $\widetilde{\E'}$ via the isomorphism 
$\mathcal H^{main} \cong \widetilde{\E'}$.    
In the last section, we give the proof of Main Theorem. 
\section{Generalities on the invariant Hilbert scheme}\label{s-generalhilb}
We review some generalities on  the invariant Hilbert scheme introduced by Alexeev and Brion in \cite{AB}. For more details refer to Brion's survey \cite{B}. 

Let $G$ be a reductive algebraic group. For any $G$-module $V$, we have its isotypical decomposition
\[
V \cong \bigoplus_{M \in \Irr(G)} \Hom^G(M,V) \otimes M,
\]
where $\Irr(G)$ stands for the set of isomorphism classes of irreducible representations of $G$.  
We call the dimension of $\Hom^G(M,V)$ the \emph{multiplicity} of $M$ in $V$. 
If the multiplicity is finite for every $M \in \Irr(G)$, we can define a function 
\[
h_V: \Irr(G) \to \mathbb Z_{\geq 0}, \quad M \mapsto h_V(M):=\dim \Hom^G(M,V),
\]
which is called the \emph{Hilbert function} of $V$. 
Let $X$ be an affine $G$-scheme of finite type, and $h$ a Hilbert function. 
The \emph{invariant Hilbert scheme} $\Hilb^G_h(X)$ associated to the triple $(G,X,h)$ is a moduli space that parametrizes $G$-stable closed subschemes of $X$ whose coordinate rings have Hilbert function $h$.  
Namely, the set-theoretical description of $\Hilb^G_h(X)$ is given as follows:
\[
\Hilb^G_h(X)=\left\{Z \subset X\; : \; 
\begin{matrix}
Z\; \mbox{is a closed}\; G\mbox{-subscheme of}\; X; \qquad \; \; \; \; \; \\
\C[Z] \cong \bigoplus_{M \in \Irr(G)} M^{\oplus h(M)}\; \mbox{as}\; G\mbox{-modules}
\end{matrix}
\right\}.
\] 
We denote by $T_{[Z]} \Hilb^G_h(X)$ the Zariski tangent space to the invariant Hilbert scheme $\Hilb^G_h(X)$ at a closed point $[Z]$. 
Let 
\[
\xymatrix{
\pi : X \ar[r] &  X \git G:=\Spec(\C[X]^G)
}
\]
 be the quotient morphism, and suppose that $X$ is irreducible. 
Then by the generic flatness theorem $\pi$ is flat over a non-empty open subset $Y_0$ of $X \git G$.  
The Hilbert function of the flat locus $\pi^{-1}(Y_0) \to Y_0$ is called the \emph{Hilbert function of the general fibers of $\pi$} and denoted by $h_X$. 
The associated Hilbert--Chow morphism 
\[
\gamma : \Hilb^G_{h_X}(X) \to X \git G,\qquad [Z] \mapsto Z \git G
\]
is an isomorphism over $Y_0$, and its restriction to the \emph{main component} $\mathcal H^{main}:=\overline{\gamma^{-1}(Y_0)}$ 
is projective and birational 
({\cite[Theorem I.1.1]{Bud}}, \cite[Proposition 3.15]{B}, 
see also \cite{Bec, Ter14, Ter}). 

To conclude this short section, we summarize Becker's idea \cite[\S 4.2]{Bec} of embedding an invariant Hilbert scheme into products of Grassmannians.  
Suppose that there is an action on $X$ by another connected reductive algebraic group $G'$.  
For any irreducible representation 
$M \in \Irr(G)$, there is a finite-dimensional $G'$-module $F_M$ that 
generates $\Hom^G(M, \C[X])$ as $\C[X]^G$-modules. 
For $[Z] \in \Hilb^G_{h_X}(X)$, we let 
\[
\xymatrix{
f_{M,Z} : F_M \ar[r] &  \Hom^G(M, \C[Z])
}
\]
 be 
the composition of the inclusion 
$F_M \hookrightarrow  \Hom^G(M, \C[X])$ and the natural surjection $\Hom^G(M, \C[X]) \to \Hom^G(M, \C[Z])$. 
Then, the quotient vector space $F_M/\Ker f_{M,Z}$ defines a point in the Grassmannian $\Gr(h_X(M), F_M^{\vee})$. 
In this way, we obtain a $G'$-equivariant morphism
\[
\eta_M : 
\Hilb^G_{h_X}(X) \to \Gr(h_X(M), F_M^{\vee}) , \qquad 
[Z] \mapsto F_M/\Ker f_{M,Z}.
\]
Furthermore, there is a finite subset $\mathcal M \subset \Irr(G)$ 
such that the morphism
\[
\xymatrix{
\gamma \times\prod_{M \in \mathcal M} \eta_M : 
\Hilb^G_{h_X}(X) \ar[r] & X \git G \times \prod_{M \in \mathcal M} \Gr(h_X(M), F_M^{\vee})
}
\]
is a closed immersion. 
\section{Generalities on spherical varieties}\label{s-spherical}
Let $G$ be a connected reductive algebraic group, and $H$ an algebraic subgroup of $G$. 
A normal $G$-variety is called \emph{spherical} if it contains a dense orbit under a Borel subgroup of $G$. 
By a \emph{spherical embedding}, we mean a normal $G$-variety $X$ together with an equivariant open embedding of a homogeneous spherical variety $G/H \hookrightarrow X$.  

Let $X$ be a spherical embedding of $G/H$ with respect to a Borel subgroup $B$. We denote by $\mathfrak X(B)$ the group of characters of $B$, and by $\C(G/H)^{(B)}$ the set of rational $B$-eigenfunctions: 
\[
\C(G/H)^{(B)}=\left\{f \in \C(G/H)^*\; :\; \exists \chi_f \in \mathfrak X(B)\; \forall g \in B \; g \cdot f= \chi_f(g)f\right\}.
\] 
Consider a homomorphism 
$\C(G/H)^{(B)} \to \mathfrak X(B)$ defined by $f \mapsto \chi_f$, and let $\Gamma \subset \mathfrak X(B)$ be its image. 
Then, $\Gamma$ is a finitely generated free abelian group, and its rank is called the \emph{rank} of $G/H$.  Since $G/H$ contains a dense $B$-orbit, the kernel of the above homomorphism  consists of constant functions.  
Therefore, we get the exact sequence
\[
1 \longrightarrow \C^* \longrightarrow \C(G/H)^{(B)} \longrightarrow \Gamma \longrightarrow 0.
\]
We see that any valuation $v : \C(G/H)^* \to \mathbb Q$ of $G/H$ defines a homomorphism $\C(G/H)^{(B)} \to \mathbb Q,\; 
f \mapsto v(f)$, 
which factors through $\Gamma$. 
Hence it induces an element 
\[
\rho_v \in Q:=\Hom(\Gamma, \mathbb Q),
\]
 namely $\rho_v(\chi_f)=v(f)$. A valuation $v$ is called \emph{$G$-invariant} if $v(g\cdot f)=v(f)$ holds for any $g \in G$, and we denote by $\mathcal V$ the set of $G$-invariant valuations. 
\begin{proposition}[{\cite[7.4 Proposition]{LV}}]
The map $\mathcal V \to Q,\; v \mapsto \rho_v$ is injective.
\end{proposition}
Let us denote by $\mathcal D(X)$ the set of $B$-stable prime divisors on $X$. 
We simply write $\mathcal D$ for $\mathcal D(G/H)$ and call  
an element of $\mathcal D$ a \emph{color}. 
If $D \in \mathcal D(X)$ non-trivially meets the open orbit $G/H$, then we have $D \cap G/H \in \mathcal D$. 
Otherwise, $D$ is an irreducible component of the complement $X\setminus (G/H)$ and hence is $G$-stable. 
Therefore, each $G$-orbit $Y$ in $X$ determines two sets
\[
\mathcal B_Y(X):=\{v_D \in \mathcal V\; :\; D \in \mathcal D_Y(X)\; \mbox{is}\; G\mbox{-stable}\}
\]
and 
\[
\mathcal F_Y(X):=\{D \cap G/H \in \mathcal D\; :\; D \in \mathcal D_Y(X)\; \mbox{is not}\; G\mbox{-stable}\},
\]
where 
\[
\mathcal D_Y(X):=\{D \in \mathcal D(X)\; :\; Y \subset D\}.
\] 
\begin{definition}
A spherical embedding $X$ is called \emph{simple} if it contains a unique closed $G$-orbit. 
\end{definition}
\begin{remark}
Any spherical embedding is covered by finitely many simple open subembeddings. 
\end{remark}
\begin{remark}\label{coordinate ring}
Let $X$ be a simple spherical embedding with a closed orbit $Y$, and set 
\[
(X)_0:=X \setminus \bigcup_{D \in \mathcal D(X) \setminus \mathcal D_Y(X)} D 
\]
and 
\[
(X)_1:=G/H \setminus \bigcup_{D \in \mathcal D \setminus \mathcal F_Y(X)} D. 
\]
Then $(X)_0$ is a $B$-stable affine open subset, and we have 
 \[
\C[(X)_0]=\{f \in \C[(X)_1]\; :\; v(f) \geq 0\; \mbox{for all}\; v \in \mathcal B_Y(X)\}.
\]
Also, we have $X=G(X)_0$. 
(see {\cite[Theorems 2.1 and 2.3]{Knop}}).
\end{remark}
Now with the preceding notation, we see that there is a natural map 
\[
\varrho: \mathcal D \to Q,\quad D \mapsto \varrho(D):=\rho_{v_D}.
\]
\begin{definition}
A \emph{colored cone} is a pair $(\mathcal C, \mathcal F)$ with 
$\mathcal C \subset Q$ and $\mathcal F \subset \mathcal D$ that satisfies the following properties:
\begin{itemize}
\item $\mathcal C$ is a cone generated by $\varrho(\mathcal F)$ and finitely many elements of $\mathcal V$;   
\item $\mathcal C^{\circ} \cap \mathcal V \neq \phi$, where $\mathcal C^{\circ}$ stands for the relative interior of $\mathcal C$. 
\end{itemize}
A colored cone $(\mathcal C, \mathcal F)$ is called \emph{strictly convex} if $\mathcal C$ is strictly convex and $0 \notin \varrho(\mathcal F)$. 
\end{definition}
Let $Y$ be a $G$-orbit in a spherical embedding $X$, and $\mathcal C_Y(X) \subset Q$ the cone generated by $\varrho(\mathcal F_Y(X))$ and $\mathcal B_Y(X)$. Then, the pair $(\mathcal C_Y(X),\mathcal F_Y(X))$ is a 
strictly convex colored cone. 
\begin{theorem}[{\cite[8.10 Proposition]{LV}}]\label{one to one}
The map $X \mapsto (\mathcal C_Y(X), \mathcal F_Y(X))$ gives a bijective correspondence between the isomorphism classes of simple spherical embeddings $X$ with a closed orbit $Y$ and strictly convex colored cones.
\end{theorem}
We say that a  pair $(\mathcal C_0,\mathcal F_0)$ is a \emph{face} of a colored cone 
$(\mathcal C, \mathcal F)$ if $\mathcal C_0$ is a face of $\mathcal C$, $\mathcal C_0^{\circ} \cap \mathcal V \neq \phi$, and $\mathcal F_0=\mathcal F \cap \varrho^{-1}(\mathcal C_0)$. 
\begin{theorem}[{\cite[Lemma 3.2]{Knop}}]\label{closure}
Let $X$ be a spherical embedding, and $Y$ a $G$-orbit. Then, the map 
$Z \mapsto (\mathcal C_Z(X), \mathcal F_Z(X))$ gives a bijective correspondence between $G$-orbits whose closure contain $Y$ and faces of $(\mathcal C_Y(X), \mathcal F_Y(X))$. 
\end{theorem}
\begin{definition}
A \emph{colored fan} is a non-empty finite set $\mathfrak F$ of colored cones satisfying the following properties: 
\begin{itemize}
\item every face of $(\mathcal C, \mathcal F) \in \mathfrak F$ belongs to $\mathfrak F$; 
\item for every $v \in \mathcal V$, there is at most one 
$(\mathcal C, \mathcal F) \in \mathfrak F$ such that  $v \in \mathcal C^{\circ}$. 
\end{itemize}
A colored fan $\mathfrak F$ is called \emph{strictly convex} if $(0, \phi) \in \mathfrak F$. 
This is equivalent to saying that all elements of $\mathfrak F$ are strictly convex.
\end{definition} 
For a spherical embedding $X$, we define 
\[
\mathfrak F(X):=\{(\mathcal C_Y(X), \mathcal F_Y(X))\; :\; Y \subset X\; \mbox{is a}\; G\mbox{-orbit}\}.
\]
Then, $\mathfrak F(X)$ is a strictly convex colored fan. 
\begin{remark}[\cite{Knop}]\label{orbit}
We can give an order relation to the set of $G$-orbits by the inclusion of closures. 
Theorems \ref{one to one} and \ref{closure} imply that $Y \mapsto (\mathcal C_Y(X), \mathcal F_Y(X))$ is an order-reversing bijection between the set of $G$-orbits and $\mathfrak F(X)$. The open orbit corresponds to $(0, \phi)$.  
\end{remark}
\begin{theorem}[{\cite[Theorem 3.3]{Knop}}]
The map $X \mapsto \mathfrak F(X)$ gives a bijective correspondence between the isomorphism classes of spherical embeddings and strictly convex colored fans. 
\end{theorem}
\begin{definition}
A spherical embedding $X$ is called \emph{toroidal} if $\mathcal F_Y(X)=\phi$ for any 
$G$-orbit $Y$. 
This is equivalent to saying that no $D \in \mathcal D$ contains a $G$-orbit in its closure.  
\end{definition}
\begin{remark}[{\cite[3.4]{BP}}, see also {\cite[\S 3.3]{Perrin}}]
\label{integralstr}
A local structure theorem for toroidal spherical embeddings implies that a toroidal spherical embedding $X$ has singularities of a toric variety with the same cones as those of $X$ and that subdividing its fan for toric varieties gives an equivariant resolution of $X$.
\end{remark}
Equivariant birational morphisms between spherical embeddings have an implication in terms of colored fans. 
\begin{theorem}[{\cite[Theorem 4.1]{Knop}}]\label{knop}
Let $X$ and $X'$ be spherical embeddings of $G/H$. 
Then, the following are equivalent. 
\begin{itemize}
\item [\rm(i)] An equivariant birational morphism $X \to X'$ exists.  
\item [\rm(ii)] For any $(\mathcal C, \mathcal F) \in \mathfrak F(X)$ there exists $(\mathcal C', \mathcal F') \in \mathfrak F(X')$ such that $\mathcal C \subset \mathcal C'$ and $\mathcal F \subset \mathcal F'$. 
\end{itemize}
\end{theorem}
In the rest of this section, we consider Weil divisors on a spherical embedding $X$. 
According to \cite{Perrin}, any Weil divisor on $X$ is linearly equivalent to a divisor of the form 
\[
\delta=\sum_{D \in \mathcal D(X)} n_D D.
\]  
\begin{theorem}[{\cite[Theorem 3.2.1]{Perrin}}]\label{cartiercriterion}
Keep the above notation. 
Then, $\delta$ is Cartier if and only if for any $G$-orbit $Y$ there exists 
$f_Y \in \C(G/H)^{(B)}$ that satisfies $n_D=v_D(f_Y)$ for any $D \in \mathcal D_Y(X)$. 
\end{theorem}
\begin{definition}[{\cite[Definition 3.2.2]{Perrin}}]
Let $X$ be a spherical embedding. 
\begin{itemize}
\item [\rm(i)] We denote by $\mathcal C(X)$ the union of all $\mathcal C_Y(X)$, where $Y$ runs over all $G$-orbits.
\item [\rm(ii)] 
A collection $l=(l_Y)$ indexed by $G$-orbits $Y$ is called a \emph{piecewise linear function} if it satisfies the following conditions:
\begin{itemize}
\item [$\bullet$] for each $G$-orbit $Y$, $l_Y$ is the restriction of an element of $\Gamma$ to $\mathcal C_Y(X)$;
\item [$\bullet$] for any $G$-orbits $Y$ and $Z$ with $Z \subset \overline{Y}$, we have $l_Z|_{
\mathcal C_Y(X)}=l_Y$. 
\end{itemize}
We denote by $PL(X)$ the abelian group consists of piecewise linear functions. 
\end{itemize}
\end{definition}
\begin{remark}[{\cite[Remark 3.2.3]{Perrin}}]
An element $l \in PL(X)$ depends only on its values on maximal cones, namely cones of closed orbits in $X$. 
\end{remark}
Let $\Car^{B}(X)$ be the group of $B$-stable Cartier divisors on a spherical embedding $X$. 
Then, we have a morphism
\[
\Car^{B}(X) \to PL(X),\quad \delta \mapsto l_{\delta},
\]
where $(l_{\delta})_Y=f_Y$ with the notation as in Theorem \ref{cartiercriterion}. 
Set
\[
\mathcal D_0(X):=\bigcup \mathcal D_Y(X),
\]
where $Y$ runs over all $G$-orbits. 
\begin{theorem}[{\cite[Theorem 17.18]{Tim}}]
\label{criterion for gg}
For any $B$-stable Cartier divisor
\[
\delta=\sum_{D \in \mathcal D_0(X)} {v_D}(l_{\delta})
 D + \sum_{D \in \mathcal D(X) \setminus \mathcal D_0(X)} n_D D
\]
on $X$, the following properties are equivalent. 
\begin{itemize}
\item [\rm(i)] The divisor $\delta$ is generated by global sections. 
\item [\rm(ii)] For any $G$-orbit $Y$, there exists $f_Y \in \C(G/H)^{(B)}$ that satisfies the following conditions:
\begin{itemize}
\item [$\bullet$] $f_Y|_{\mathcal C_Y(X)}=l_{\delta}|_{\mathcal C_Y(X)}$;
\item [$\bullet$] $f_Y|_{\mathcal C(X) \setminus \mathcal C_Y(X)} \leq l_{\delta}|_{\mathcal C(X) \setminus \mathcal C_Y(X)}$; 
\item [$\bullet$] ${v_D}(f_Y) \leq n_D$ for any $D \in \mathcal D(X) \setminus \mathcal D_0(X)$.
\end{itemize}
\end{itemize}
\end{theorem}
\section{Quasihomogeneous $SL(2)$-varieties and their spherical geometry}\label{s-popovvar}
In \cite{P}, Popov gives a complete classification of affine normal quasihomogeneous $SL(2)$-varieties. 
Consult also the book of Kraft \cite{K}. 
\begin{theorem}[\cite{P}]\label{popov}
Every 3-dimensional affine normal quasihomogeneous $SL(2)$-variety containing more than one orbit is uniquely determined by a pair of numbers 
$(l,m) \in \{\mathbb Q \cap (0,1]\} \times \mathbb N$. 
\end{theorem}
We denote the corresponding variety by $E_{l,m}$. The numbers $l$ 
and $m$ are called the \emph{height} and the \emph{degree} of $E_{l,m}$,
 respectively. 
Write $l=p/q$, where $g.c.d.(q, p)=1$. 
\begin{theorem}[\cite{G}, see also {\cite[Corollary 2.7]{BH}}]\label{toric}
An affine normal quasihomogeneous $SL(2)$-variety $E_{l,m}$ is toric 
if and only if $q-p$ divides $m$. 
\end{theorem}
We use the following notation for some closed subgroups of $SL(2)$:
\begin{align*}
&T:=\left\{
\begin{pmatrix}
t & 0 \\
0 & t^{-1}
\end{pmatrix}
 :\;
t \in \C^*
\right\};\quad 
B:=\left\{
\begin{pmatrix}
t & u \\
0 & t^{-1}
\end{pmatrix}:\; 
t \in \C^*,\; u \in \C
\right\}; \\
&U_n:=\left\{
\begin{pmatrix}
\zeta & u \\
0 & \zeta^{-1}
\end{pmatrix}
:\; 
\zeta^n=1,\; u \in \C
\right\}; \quad  
C_n:=\left\{
\begin{pmatrix}
\zeta & 0 \\
0 & \zeta^{-1}
\end{pmatrix}
:\; 
\zeta^n=1
\right\}.
\end{align*}
An $SL(2)$-variety $E_{l,m}$ is smooth if and only if $l=1$ (see \cite{P}). 
If $l <1$, then $E_{l,m}$ contains three $SL(2)$-orbits: 
the open orbit $\mathfrak U$, a 2-dimensional orbit $\mathfrak D$, and the closed orbit 
$\{O\}$. The fixed point $O$ is a unique $SL(2)$-invariant singular 
point.  
Let 
\begin{equation}
k:=g.c.d.(m, q-p),\quad  a:=\frac{m}{k},\quad  b:=\frac{q-p}{k}.
\label{akb}
\end{equation}
Then we have 
\[
\mathfrak U \cong SL(2)/C_m,\quad  \mathfrak D \cong SL(2)/U_{a(q+p)}.
\]
\begin{remark}\label{explicit}
An explicit construction of the variety $E_{l,m}$ is reduced to determine a system 
of generators of the following semigroup (see \cite{K}, \cite{Pa}):
\[
M^+_{l,m}:=\left\{ (i,j) \in \mathbb Z^2_{\geq 0}\;  :\;  j \leq li,\; m|(i-j)\right\}.
\]
Let $(i_1, j_1),\; \dots,\; (i_u, j_u)$ be a system of generators of 
$M^+_{l,m}$, and consider a vector
\[
v=(X^{i_1}Y^{j_1},\; \dots,\; X^{i_u}Y^{j_u}) \in V(i_1+j_1) \oplus \dots 
\oplus V(i_u+j_u),
\]
where $V(n):=\Sym^n \langle X, Y \rangle$ is the irreducible $SL(2)$-representation of highest weight $n$. Then, $E_{l,m}$ is isomorphic to the closure $\overline{SL(2) \cdot v} \subset V(i_1+j_1) \oplus \dots 
\oplus V(i_u+j_u)$.  
\begin{figure}[h]
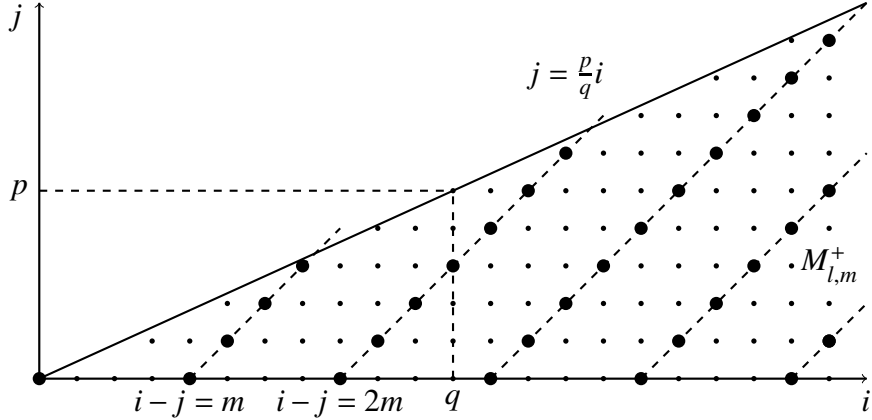

\begin{center}
\tikz{
(-1,-1) grid (11,5);
\draw[thick](0,0)--(0,5);
\draw[thick](0,0)--(11,0);
\draw[thick](0,0)--(11,5);
\draw node at(7,4){$j=\frac{p}{q}i$};
\draw node at(10.5,1.5){$M^+_{l,m}$};
\draw node at(11,-0.3){$i$};
\draw node at(-0.3,4.8){$j$};
\coordinate (J) at (0,5);
\coordinate (I) at (11,0);
\coordinate (O) at (0,0);
\draw[->] (O) -- (J);
\draw[->] (O) -- (I);
{\draw[dashed][thick](2,0)--(4,2)};
{\draw[dashed][thick](4,0)--(7.5,3.5)};
{\draw[dashed][thick](6,0)--(11,5)};
{\draw[dashed][thick](8,0)--(11,3)};
{\draw[dashed][thick](10,0)--(11,1)};
{\draw[dashed][thick](5.5,0)--(5.5,2.5)};
{\draw[dashed][thick](0,2.5)--(5.5,2.5)};
\coordinate (P) at (5.5,0);
 \path (P) node [below] {$q$};
\coordinate (Q) at (0,2.5);
 \path (Q) node [left] {$p$};
\coordinate (A) at (2,0);
       {\fill (A) circle (2.5pt)};
       \path (A) node [below] {$i-j=m$};
\coordinate (B) at (4,0);
       {\fill (B) circle (2.5pt)};
       \path (B) node [below] {$i-j=2m$};
\coordinate (C) at (6,0);
       {\fill (C) circle (2.5pt)};
\coordinate (D) at (8,0);
       {\fill (D) circle (2.5pt)};
\coordinate (E) at (10,0);
       {\fill (E) circle (2.5pt)};
{\fill (0,0) circle (2.5pt)};
{\fill (0.5,0) circle (1pt)};
{\fill (1,0) circle (1pt)};       
{\fill (1.5,0) circle (1pt)};
{\fill (1.5,0.5) circle (1pt)};
{\fill (2,0.5) circle (1pt)};
{\fill (2.5,0) circle (1pt)};
{\fill (2.5,1) circle (1pt)};
{\fill (3,0) circle (1pt)};
{\fill (3,0.5) circle (1pt)};
{\fill (3.5,0) circle (1pt)};
{\fill (3.5,0.5) circle (1pt)};
{\fill (3.5,1) circle (1pt)};
{\fill (3.5,1.5) circle (2.5pt)};
{\fill (4,0.5) circle (1pt)};
{\fill (4,1) circle (1pt)};
{\fill (4,1.5) circle (1pt)};
{\fill (4.5,0) circle (1pt)};
{\fill (4.5,1) circle (1pt)};
{\fill (4.5,1.5) circle (1pt)};
{\fill (4.5,2) circle (1pt)};
{\fill (5,0) circle (1pt)};
{\fill (5,0.5) circle (1pt)};
{\fill (5,1) circle (1pt)};
{\fill (5,1.5) circle (1pt)};
{\fill (5,2) circle (1pt)};
{\fill (5.5,0) circle (1pt)};
{\fill (5.5,0.5) circle (1pt)};
{\fill (5.5,1) circle (1pt)};
{\fill (5.5,1.5) circle (1pt)};
{\fill (5.5,2) circle (1pt)};
{\fill (5.5,2.5) circle (1pt)};
{\fill (6,0.5) circle (1pt)};
{\fill (6,1) circle (1pt)};
{\fill (6,1.5) circle (1pt)};
{\fill (6,2) circle (1pt)};
{\fill (6,2.5) circle (1pt)};
{\fill (6.5,0) circle (1pt)};
{\fill (6.5,0.5) circle (1pt)};
{\fill (6.5,1) circle (1pt)};
{\fill (6.5,1.5) circle (1pt)};
{\fill (6.5,2) circle (1pt)};
{\fill (6.5,2.5) circle (1pt)};
{\fill (7,0) circle (1pt)};
{\fill (7,0.5) circle (1pt)};
{\fill (7,1) circle (1pt)};
{\fill (7,1.5) circle (1pt)};
{\fill (7,2) circle (1pt)};
{\fill (7,2.5) circle (1pt)};
{\fill (7,3) circle (2.5pt)};
{\fill (7.5,0) circle (1pt)};
{\fill (7.5,0.5) circle (1pt)};
{\fill (7.5,1) circle (1pt)};
{\fill (7.5,1.5) circle (1pt)};
{\fill (7.5,2) circle (1pt)};
{\fill (7.5,2.5) circle (1pt)};
{\fill (7.5,3) circle (1pt)};
{\fill (8,0.5) circle (1pt)};
{\fill (8,1) circle (1pt)};
{\fill (8,1.5) circle (1pt)};
{\fill (8,2) circle (1pt)};
{\fill (8,2.5) circle (1pt)};
{\fill (8,3) circle (1pt)};
{\fill (8,3.5) circle (1pt)};
{\fill (8.5,0) circle (1pt)};
{\fill (8.5,0.5) circle (1pt)};
{\fill (8.5,1) circle (1pt)};
{\fill (8.5,1.5) circle (1pt)};
{\fill (8.5,2) circle (1pt)};
{\fill (8.5,2.5) circle (1pt)};
{\fill (8.5,3) circle (1pt)};
{\fill (8.5,3.5) circle (1pt)};
{\fill (9,0) circle (1pt)};
{\fill (9,0.5) circle (1pt)};
{\fill (9,1) circle (1pt)};
{\fill (9,1.5) circle (1pt)};
{\fill (9,2) circle (1pt)};
{\fill (9,2.5) circle (1pt)};
{\fill (9,3) circle (1pt)};
{\fill (9,3.5) circle (1pt)};
{\fill (9,4) circle (1pt)};
{\fill (9.5,0) circle (1pt)};
{\fill (9.5,0.5) circle (1pt)};
{\fill (9.5,1) circle (1pt)};
{\fill (9.5,1.5) circle (1pt)};
{\fill (9.5,2) circle (1pt)};
{\fill (9.5,2.5) circle (1pt)};
{\fill (9.5,3) circle (1pt)};
{\fill (9.5,3.5) circle (1pt)};
{\fill (9.5,4) circle (1pt)};
{\fill (10,0.5) circle (1pt)};
{\fill (10,1) circle (1pt)};
{\fill (10,1.5) circle (1pt)};
{\fill (10,2) circle (1pt)};
{\fill (10,2.5) circle (1pt)};
{\fill (10,3) circle (1pt)};
{\fill (10,3.5) circle (1pt)};
{\fill (10,4) circle (1pt)};
{\fill (10,4.5) circle (1pt)};
{\fill (10.5,0) circle (1pt)};
{\fill (10.5,0.5) circle (2.5pt)};
{\fill (10.5,1) circle (1pt)};
{\fill (10.5,2) circle (1pt)};
{\fill (10.5,2.5) circle (1pt)};
{\fill (10.5,3) circle (1pt)};
{\fill (10.5,3.5) circle (1pt)};
{\fill (10.5,4) circle (1pt)};
{\fill (10.5,4.5) circle (2.5pt)};
{\fill (2.5,0.5) circle (2.5pt)};
{\fill (3,1) circle (2.5pt)};
{\fill (4.5,0.5) circle (2.5pt)};
{\fill (5,1) circle (2.5pt)};
{\fill (5.5,1.5) circle (2.5pt)};
{\fill (6,2) circle (2.5pt)};
{\fill (6.5,2.5) circle (2.5pt)};
{\fill (6.5,0.5) circle (2.5pt)};
{\fill (7,1) circle (2.5pt)};
{\fill (7.5,1.5) circle (2.5pt)};
{\fill (8,2) circle (2.5pt)};
{\fill (8.5,2.5) circle (2.5pt)};
{\fill (9,3) circle (2.5pt)};
{\fill (9.5,3.5) circle (2.5pt)};
{\fill (10,4) circle (2.5pt)};
{\fill (8.5,0.5) circle (2.5pt)};
{\fill (9,1) circle (2.5pt)};
{\fill (9.5,1.5) circle (2.5pt)};
{\fill (10,2) circle (2.5pt)};
{\fill (10.5,2.5) circle (2.5pt)};
{\fill (10.5,0.5) circle (2.5pt)};
}
\end{center}
\caption{The semigroup $M^+_{l,m}$}
\end{figure}
\end{remark}
According to \cite[\S 1]{BH}, an affine normal quasihomogeneous $SL(2)$-variety $E_{l,m}$ has a description as a categorical quotient 
of a hypersurface in $\C^5$. 
We consider $\C^5$ as the $SL(2)$-module $V(0) \oplus 
V(1) \oplus V(1)$ with coordinates $X_0,\; X_1,\; X_2,\; X_3,\; X_4$, and 
identify $X_1,\; X_2,\; X_3,\; X_4$ with the coefficients of the $2 \times 2$ matrix 
\[
\begin{pmatrix}
X_1 & X_3 \\
X_2 & X_4\\
\end{pmatrix}
\] 
so that $SL(2)$ acts by left multiplication. 
We moreover consider actions of the following two diagonalizable groups:
\begin{align*}
& G_0:=\left\{\diag(t,\; t^{-p},\; t^{-p},\; t^q,\; t^q)\; :\; t \in \C^*\right\} \cong 
\C^*;\\
& G_m:=\left\{\diag(1,\; \zeta^{-1},\; \zeta^{-1},\; \zeta,\; \zeta)\; :\; \zeta^m=1\right\} \cong 
\mu_m.
\end{align*} 
It is easy to see that the $SL(2)$-action on $\C^5$ commutes with the $\G$-action. 
\begin{theorem}[{\cite[Theorem 1.6]{BH}}]\label{batyrev}
Let $E_{l,m}$ be a $3$-dimensional affine normal quasihomogeneous $SL(2)$-variety of height 
$l=p/q$ and degree $m$. Then, $E_{l,m}$ is isomorphic to the categorical quotient 
of the affine hypersurface 
\[
\C^5 \supset H_{q-p}:=(X_0^{q-p}=X_1X_4-X_2X_3)
\]
modulo the action of $\G$. 
\end{theorem}
\begin{remark}\label{monomial}
According to the proof of \cite[Theorem 1.6]{BH}, the dense open orbit $\mathfrak U$ in $\E$ is isomorphic to the $\G$-quotient of the open subset in $H_{q-p}$ defined by the condition $X_0 \neq 0$.  
Also, the ring of $G_0$-invariants of $H_{q-p} \cap \{X_0 \neq 0\}$ is generated by the monomials 
\[
X:=X_0^pX_1,\quad  Y:=X_0^{-q}X_3,\quad Z:=X_0^pX_2,\quad W:=X_0^{-q}X_4, 
\]
wich satisfy the equation 
\[
\det
\begin{pmatrix}
X & Y \\
Z & W
\end{pmatrix}
=X_0^{p-q}X_1X_4-X_0^{p-q}X_2X_3=1. 
\] 
\end{remark}
An $SL(2)$-variety $\E$ has another description as an affine categorical quotient. 
To see this, let $H_b \subset \C^5$ be an affine hypersurface defined by the equation 
\[
Y_0^b=X_1X_4-X_2X_3,
\] 
and consider the action of the group $G_0' \times G_a$, where
\begin{align*}
& G_0':=\left\{\diag(t^k,\; t^{-p},\; t^{-p},\; t^q,\; t^q)\; : \; t \in \C^*\right\} \cong \C^*,\\
& G_a:=\left\{\diag(1,\; \zeta^{-1},\; \zeta^{-1},\; \zeta,\; \zeta)\; : \; \zeta^a=1\right\} \cong \mu_a.
\end{align*}
\begin{theorem}[{\cite[Theorem 1.7]{BH}}]\label{batyrev2}
Let $\E$ be a $3$-dimensional affine normal quasihomogeneous $SL(2)$-variety of height $l=p/q$ and  degree $m$. 
Then, $\E$ is isomorphic to the categorical quotient of $H_b$ modulo the action of $G_0' \times G_a$. 
\end{theorem}
Let $L^-$ and $L^+$ be linearizations of the trivial line bundle over $H_b$ corresponding to the non-trivial characters
\[
\chi^- : G_0' \times G_a \to \C^*,\quad  (t, \zeta) \mapsto t^{k-p+q}
\]
and 
\[
\chi^+ : G_0' \times G_a \to \C^*,\quad (t, \zeta) \mapsto t^{-k+p-q}
\]
of $G_0' \times G_a$, respectively, and  
consider the following Zariski open subsets of $H_b$: 
\begin{equation*}
U^-:=H_b \setminus \{X_3=X_4=0\}, \quad U^+:=H_b \setminus \{X_1=X_2=0\}.
\end{equation*}
\begin{theorem}[{\cite[Propositions 3.2 and 3.3]{BH}}]
The subsets $H_b^{ss}(L^-)$ and $H_b^{ss}(L^+)$ of semistable points of $H_b$ with respect to the $G_0' \times G_a$-linearized line bundles $L^-$ and $L^+$ are $U^-$ and $U^+$, respectively. 
\end{theorem}
\begin{theorem}[{\cite[Theorem 3.4]{BH}}]\label{flip}
With the above notation, set 
\[
\E^-:=H_b^{ss}(L^-) \git (G_0' \times G_a), \quad 
\E^+:=H_b^{ss}(L^+) \git (G_0' \times G_a). 
\]
Then, the open embeddings $H_b^{ss}(L^-)\subset H_b$ and $H_b^{ss}(L^+)\subset H_b$ define natural birational morphisms $\E^- \to \E$ and 
$\E^+ \to \E$, and the $SL(2)$-equivariant flip
\[
\xymatrix@R=5pt{
\E^- \ar@{.>}[rr] \ar[rd] & & \E^+ \ar[ld] \\
& \E &
}
\]
\end{theorem}
\begin{remark}[{\cite[Remarks 3.12 and 4.2]{BH}}]\label{clembedding}
Let $\E \hookrightarrow V \cong V(i_1+j_1) \oplus \dots \oplus V(i_u+j_u)$ be an equivariant closed embedding (see Remark \ref{explicit}), and consider an action of $t \in \C^*$ on $V$ defined by multiplication of $(t^{i_1-j_1},\; \dots,\; t^{i_u-j_u})$. 
Then, since this $\C^*$-action commutes with the $SL(2)$-action, an affine variety $\E \subset V$ remains stable under the $\C^*$-action, and this enables us to consider $\E$ as an $SL(2) \times \C^*$-variety. 
We remark that there is another way to define the same $\C^*$-action on $\E$:  we 
consider an action of  $\C^*$ on $H_b$ defined by the matrices
\[
\diag(1,\; s^{-1},\; s^{-1},\; s,\; s),\quad s \in \C^*.
\]
Since this $\C^*$-action commutes with the $SL(2) \times G_0' \times G_a$-action, it descends to $\E$, and we see that this $\C^*$-action coincides with the one defined above. 
\end{remark}
\begin{theorem}[{\cite[Proposition 4.1]{BH}}]\label{spherical E}
An affine $SL(2) \times \C^*$-variety $\E$ is spherical with respect to the Borel subgroup $\tilde{B}:=B \times \C^*$. 
\end{theorem}
Let 
$\E':=Bl^{\omega}_O(\E)$ 
be the weighted blow-up of $\E$ with weight $\omega$ defined by the $\C^*$-action considered in Remark \ref{clembedding}. Then we obtain surjective morphisms $\E' \to \E^-$ and  $\E' \to \E^+$ such that the following diagram commutes:
\[
\xymatrix@R=5pt{
& \E' \ar[ld] \ar[rd] & \\ 
\E^-  \ar@{.>}[rr] \ar[rd] & & \E^+ \ar[ld] \\
& \E &
}
\]
\begin{theorem}[{\cite[\S 3]{BH}}]\label{local}
The weighted blow-up $\E'$ contains a unique closed $SL(2) \times \C^*$-orbit $C$ isomorphic to $\P^1$. 
Moreover, along the closed orbit $C$, the variety $\E'$ is locally isomorphic to $\C \times \C^2/\mu_b$. 
\end{theorem}
\begin{remark}
In view of Theorem \ref{toric}, $\E'$ is smooth 
if and only if $\E$ is toric, and in the toric case the weight $\omega$ is trivial. 
\end{remark}
Batyrev and Haddad compute the colored cones of the simple spherical varieties $\E$, $\E^-$, $\E^+$, and $\E'$ (see \cite[\S 4]{BH}). 
Firstly, the lattice $\Gamma$ of rational $\tilde{B}$-eigenfunctions on $\mathfrak U$ is given as follows:
\[
\Gamma \cong \{Z^iW^j \in \C(\mathfrak U)^*\; :\; m|(i-j)\}.
\]
The varieties $\E$, $\E^-$, and $\E^+$ contain exactly three $\tilde{B}$-stable divisors
\[
D:=(H_{b} \cap \{Y_0=0\}) \git (G_0' \times G_a),
\]
\[
S^-:=(H_{b} \cap \{X_4=0\}) \git (G_0' \times G_a),
\]
and 
\[
S^+:=(H_{b} \cap \{X_2=0\}) \git (G_0' \times G_a),
\]
and $\E'$ contains one more $SL(2) \times \C^*$-stable divisor $D' \cong \mathbb P^1 \times \mathbb P^1$, the exceptional divisor of the blow-up $\E' \to \E$. 
The divisors $D$, $S^-$, $S^+$, and $D'$ define lattice vectors $\rho_{v_D},\; \rho_{v_{S^-}},\; \rho_{v_{S^+}},\; \rho_{v_{D'}} \in \Gamma^{\vee}$ in the dual space $Q=\Hom(\Gamma, \mathbb Q)$, and we can consider $\{\rho_{v_{S^-}}, \rho_{v_{S^+}}\}$ as a $\mathbb Q$-basis of $Q$. The set $\mathcal V$ of $SL(2) \times \C^*$-invariant valuations is given as $\mathcal V=\{x \rho_{v_{S^+}}+y \rho_{v_{S^-}}\in Q\; : \; x+y \leq 0\}$, and  the colored cones of $\E$, $\E^-$, $\E^+$, and $\E'$ are described as follows: 
\begin{align*}
\mathcal C:=\mathcal C(\E)=\mathbb Q_{\geq 0} \rho_{v_{D}} + \mathbb Q_{\geq 0} \rho_{v_{S^-}}, & \quad  
\mathcal F:=\mathcal F(\E)=\{\rho_{v_{S^+}}, \rho_{v_{S^-}}\} \\
 \mathcal C^-:=\mathcal C(\E^-)=\mathbb Q_{\geq 0} \rho_{v_{D}} + \mathbb Q_{\geq 0} \rho_{v_{S^+}},& \quad  
\mathcal F^-:=\mathcal F(\E^-)=\{\rho_{v_{S^+}}\} \\
 \mathcal C^+:=\mathcal C(\E^+)=\mathbb Q_{\geq 0} \rho_{v_{D}} + \mathbb Q_{\geq 0} \rho_{v_{S^-}},& \quad  
\mathcal F^+:=\mathcal F(\E^+)=\{\rho_{v_{S^-}}\} \\
 \mathcal C':=\mathcal C(\E')=\mathbb Q_{\geq 0} \rho_{v_{D}} + \mathbb Q_{\geq 0} \rho_{v_{D'}},& \quad 
\mathcal F':=\mathcal F(\E')=\phi.
\end{align*}
\section{Statement of the main result}\label{s-flat}
In this section, we first review some facts from our previous article \cite{Ku} that hold without the toric hypothesis. 
Let 
\[
\xymatrix{
\pi : H_{q-p} \ar[r] & H_{q-p}\git (\G) \cong E_{l,m}
}
\] 
be the quotient morphism. Then, $\pi$ is flat over the 
open orbit $\mathfrak U \subset \E$, and the Hilbert function $h:=h_{H_{q-p}}$ of the general fibers of $\pi$ coincides with that of the regular representation 
$\C[\G]$:  
\[
h : \Irr(\G) \cong \mathbb Z \times \mathbb Z/m \mathbb Z \to \mathbb Z_{\geq 0},\quad (n,d) \mapsto h(n,d)=1. 
\]
The Hilbert--Chow morphism 
\[
\xymatrix{
\gamma : \mathcal H=\Hilb^{\G}_h(H_{q-p}) \ar[r] & H_{q-p} \git (\G) \cong \E
}
\]
is an isomorphism over the open orbit $\mathfrak U$, and the main component $\mathcal H^{main}$ is the Zariski closure $\overline{\gamma^{-1}(\mathfrak U)}$. 
\begin{remark}
For a $\G$-module $V$, we denote by $V_{(n,d)}$ the weight space of weight $(n,d) \in \Z \times \Z/m \Z$ as in \cite[Remark 3.2.1]{Ku}. 
\end{remark}
Let $A$ be the polynomial ring $\C[X_0, X_1, X_2, X_3, X_4]$, and consider the following ideals of $A$: 
\[
I_1:=(X_0^{q-p}-X_1X_4,\; X_2,\; X_3,\; 1-X_0^{mp}X_1^{m}); 
\]
\[
I_0:=(X_0^{q-p}-X_1X_4,\; X_2,\; X_3,\; X_0^{mp}X_1^{m}).
\]
\begin{theorem}[{\cite[\S 4]
{Ku}}]\label{idealU}
The following properties are true. 
\begin{itemize}
\item [\rm(i)] The quotient rings $A/I_1$ and $A/I_0$ have Hilbert function $h$, namely we have $\dim (A/I_0)_{(n,d)}=\dim (A/I_1)_{(n,d)}=h(n,d)$ for any $(n,d) \in \Z \times \Z/m \Z$. 
\item [\rm(ii)] The $SL(2) \times \C^*$-equivariant isomorphism 
$\gamma|_{\gamma^{-1}(\mathfrak U)} : \gamma^{-1}(\mathfrak U) \to \mathfrak U$ is 
given by sending $[I_1]$ to $\pi(x)$, where $x=(1,1,0,0,1) \in H_{q-p}$.
\item [\rm(iii)] The closed point $[I_0]$ is contained in the singular fiber $\gamma^{-1}(O)$. 
\end{itemize}
\end{theorem}
Let $S$ be the coordinate ring of $H_{q-p}$:
\[
S:=\C[H_{q-p}] \cong A/(X_0^{q-p}-X_1X_4+X_2X_3).
\]
For any weight $(n,d) \in \mathbb Z \times \mathbb Z/ m \mathbb Z$, there is a finite-dimensional $SL(2) \times \C^*$-module 
$F_{n,d}$ that generates the weight space $S_{(n,d)}$ over the invariant ring $S^{\G}$. 
By \cite[Lemma 4.3]{Ku}, we can take $F_{-p, -1}
=\langle X_1, X_2 \rangle$ and 
$F_{q, 1}
=\langle X_3, X_4 \rangle$. 
It follows that for any closed point $[I] \in \mathcal H$, we have 
\begin{equation}
s_1X_1+s_2X_2 \in I \label{s1s2}
\end{equation}
and 
\begin{equation}
s_3X_3+s_4X_4 \in I \label{s3s4}
\end{equation}
for some $(s_1, s_2) \neq 0$ and $(s_3, s_4 ) \neq 0$, respectively. 
Therefore, 
we can construct the following equivariant morphisms: 
\[
\xymatrix{
\eta_{-p, -1} : \mathcal H \ar[r] & \Gr(1, F_{-p, -1}^{\vee}) \cong \mathbb P^1}, \quad 
\xymatrix{\eta_{q,1} : \mathcal H \ar[r] & \Gr(1, F_{q,1}^{\vee}) \cong \mathbb P^1}.
\]
Set  
\[
\xymatrix{
\psi:=\gamma \times \eta_{-p,-1} \times \eta_{q,1}
 :\mathcal H \ar[r] & \E \times \P^1 \times \P^1.
}
 \] 
Then we have 
\begin{equation}
\psi([I_1])=(\pi(x), [1:0], [0:1]) \label{point}
\end{equation}
by its construction (see \cite[\S 6]{Ku}). 

In what follows, we show that the restriction $\gamma|_{\mathcal H^{main}}$ factors through the weighted blow-up $\E'$ and that $\psi(\mathcal H^{main}) \cong \E'$. 
First, notice that we have the following equivariant commutative diagram:
\begin{equation*}
\xymatrix@R=6pt{
& & \E' \ar[rrddd] \ar[dd] \ar[llddd] & &\\
& & & & \\
& & \E^- \times_{\E} \E^+ \ar[rrd] \ar[lld]& &\\
\E^- \ar[rrd] \ar@{.>}[rrrr]& & & & \E^+ \ar[lld]\\
& & \E & &
} \label{fiber}
\end{equation*}
\begin{lemma}\label{fibered product}
We have $\E' \cong \E^- \times_{\E} \E^+$ as spherical $SL(2) \times \C^*$-varieties.
\end{lemma}
\begin{proof}
Set $\E''=\E^- \times_{\E} \E^+$. 
Then, $\E''$ is a simple spherical $SL(2) \times \C^*$-variety with a dense orbit isomorphic to $\mathfrak U$. 
Let $(\mathcal C'',\mathcal F'')$ be the colored cone of $\E''$. 
Then we have  $\mathcal C'' \subset \mathcal C^{-}$, 
$\mathcal C' \subset \mathcal C''$, $\mathcal F'' \subset \mathcal F^{+}$, and $\mathcal F'' \subset \mathcal F^{-}$ by Theorem \ref{knop}. This implies that $\mathcal F''=\phi=\mathcal F'$.
Since $\mathcal C''$ is generated by $\varrho(\mathcal F'')$ and finite elements of $\mathcal V$, we obtain $\mathcal C'' \subset \mathcal V$. This yields that $\mathcal C' = \mathcal C''$, and 
hence we have $(\mathcal C'', \mathcal F'')=(\mathcal C', \mathcal F')$. 
Therefore, $\E'' \cong \E'$ by Theorem \ref{one to one}. 
\end{proof}
\begin{lemma}\label{closed embedding}
There are $SL(2) \times \C^*$-equivariant embeddings:
\[
\E^+ \hookrightarrow \E \times \Gr(1, F^{\vee}_{-p,-1}) \cong \E \times \P^1, \quad 
\E^- \hookrightarrow \E \times \Gr(1, F^{\vee}_{q,1}) \cong \E \times \P^1.
\]
\end{lemma}
\begin{proof} 
We have the following equivariant morphism (this morphism was first constructed in the proof of \cite[Theorem 3.10]{BH}):
\begin{equation*}
U^+ \to \Gr(1, F^{\vee}_{-p,-1}) \cong \P^1, \quad (Y_0, X_1, X_2, X_3, X_4) \mapsto [X_1: X_2].
\end{equation*}
Also, we have an equivariant morphism $U^+ \to \E$ as a composition of the inclusion $U^+ \hookrightarrow H_b$ and the quotient morphism $H_b \to \E$. Therefore, we get a $G_0' \times G_a$-invariant morphism $U^+ \to \E \times \P^1$, which factors through $\E^+$: 
\[
\xymatrix{
U^+ \ar[r] \ar[d] & \E \times \P^1 \\
U^+ \git (G_0' \times G_a) = \E^+ \ar[ru]_{\quad \alpha^+} & 
}
\]
Let $[T_1: T_2]$ be the coordinate of $\P^1$. Then for each $i \in \{1, 2\}$, we have the following commutative diagram:
\[
\xymatrix{
U^+ \cap \{X_i \neq 0\}=H_b \cap \{X_i \neq 0\} \ar[r] \ar[d] & \Spec\left(\C[\E]\left[\frac{T_1}{T_i}, \frac{T_2}{T_i}\right]\right)\\
(H_b \cap \{X_i \neq 0\}) \git (G_0' \times G_a) \ar[ru]_{\quad {\alpha^+}{|_{ \{X_i \neq 0\}}}} & 
}
\]
We see that
\begin{equation*}
(\C[H_b]_{X_i})^{G_0' \times G_a}=\C[H_b]^{G_0' \times G_a}\left[\frac{X_1}{X_i}, \frac{X_2}{X_i}\right] \label{subring}
\end{equation*}
holds as a subring of $\C[H_b]_{X_i}$, and therefore $\alpha^+$ is a closed immersion. 
%
Analogously, we have an equivariant morphism 
\begin{equation*}
U^- \to \Gr(1, F^{\vee}_{q,1}) \cong \P^1, \quad (Y_0, X_1, X_2, X_3, X_4) \mapsto [X_3: X_4],
\end{equation*}
which induces an equivariant morphism 
$\alpha^- : \E^- \to \E \times \mathbb P^1$. 
In a similar way, we see that $\alpha^-$ is a closed immersion. 
\end{proof}
By Lemmas \ref{fibered product} and \ref{closed embedding}, we get the following equivariant closed embedding: 
\[
\xymatrix{
\varphi : \E' \cong \E^- \times _{\E} \E^+ \ar@{^{(}-_>}[r] &  \E \times \P^1 \times \P^1.
}
\] 
\begin{corollary}\label{iota}
We have $\psi(\mathcal H^{main}) \cong \E'$. 
\end{corollary}
\begin{proof}
Let $x=(1,1,0,0,1) \in H_{q-p}$. Then by the construction of $\varphi$, the $SL(2) \times \C^*$-orbit of $(\pi(x), [1:0], [0:1])$ is the dense open orbit in $\varphi(\E')$ isomorphic to $\mathfrak U$. Taking Theorem \ref{idealU} (ii) and \eqref{point} into account, we get $\psi(\mathcal H^{main}) =\overline{\psi(\gamma^{-1}(\mathfrak U))} \cong \E'$. 
\end{proof}
Summarizing, we obtain the following equivariant commutative diagram:
\[
\xymatrix{
\mathcal H^{main} \ar[rr]^{\psi |_{\mathcal H^{main}}\qquad \; \; \;} \ar[rrd]_{\gamma|_{\mathcal H^{main}}}& & \E' \subset \E \times \mathbb P^1 \times \mathbb P^1 \ar[d]\\
& &  \E
}
\]

We have seen in Remark \ref{integralstr} that every toroidal spherical 
variety has an equivariant resolution of singularities given by 
subdividing its fan for toric varieties. 
We apply this to the simple toroidal spherical $SL(2) \times \C^*$-variety $\E'$ and describe the minimal resolution of $\E'$ in terms of its colored fan. 
Firstly, we can take $\{(2,0),\; (m,m)\}$ as a basis of the lattice $\Gamma \subset \mathfrak X(\tilde{B}) \cong \mathbb Z^2$. Let us denote its dual basis by $\{\mbox{\bf u}_1,\; \mbox{\bf u}_2\}$. 
By virtue of \cite[Theorem 2]{Pan} and \cite[Proposition 2.8]{BH}, we see that 
\[
\rho_{v_{D}}=-b\mbox{\bf u}_1+ap\mbox{\bf u}_2, \quad \rho_{v_{{S^-}}}=\mbox{\bf u}_1, \quad \rho_{v_{{S^+}}}=\mbox{\bf u}_1+m\mbox{\bf u}_2, \quad \rho_{v_{D'}}=\mbox{\bf u}_2.
\] 
Therefore, $\E'$ has singularities of an affine toric surface defined by the following cone (see \cite[Remark 3.12]{BH}): 
\[
\sigma:=\mathbb Q_{\geq 0}\mbox{\bf u}_2 +\mathbb Q_{\geq 0}(-b\mbox{\bf u}_1+ap\mbox{\bf u}_2). 
\] 
Let $\alpha$ and $\beta$ be the quotient and the remainder of $mp$ divided by $q-p$, respectively, i.e.,  
\begin{equation}
mp=\alpha(q-p)+\beta, \label{mp}
\end{equation}
and set  
\begin{equation}
t:=\frac{q-p-\beta}{k}=(\alpha+1)b-ap. \label{t}
\end{equation} 
We consider the base change 
\begin{equation*}
\begin{pmatrix}
\mbox{\bf u}_1'\\
\mbox{\bf u}_2'
\end{pmatrix}
:=
\begin{pmatrix}
-1 & \alpha+1 \\
0 & 1
\end{pmatrix}
\begin{pmatrix}
\mbox{\bf u}_1\\
\mbox{\bf u}_2
\end{pmatrix}
\end{equation*}
to make $\sigma$ into the normal form (see \cite[Proposition 10.1.1]{CLS}):
\[
\sigma=\mathbb Q_{\geq 0}\mbox{\bf u}_2'+ \mathbb Q_{\geq 0}(b \mbox{\bf u}_1'-t \mbox{\bf u}_2').
\]
Therefore, the toric variety of the cone $\sigma$ is a cyclic quotient singularity of type $\frac{1}{b}(1, t)$, and it has a minimal resolution described by the Hirzebruch--Jung continued fraction of $b/t$ (see \cite{CLS}, \cite{F}): 
\begin{equation*}
\frac{b}{t}=
c_1 - \cfrac{1}{c_2 -
         \cfrac{1}{\ldots -
         \cfrac{1}{c_r}}}. 
\end{equation*}
Set 
\begin{align*}
 P_0:=0, \quad  &Q_0:=-1,\\
 P_1:=1, \quad  &Q_1:=0. 
 \end{align*}
 For $2 \leq i \leq r+1$, we recursively define 
 \begin{equation}
 P_i:=c_{i-1} P_{i-1}-P_{i-2}, \quad   Q_i:=c_{i-1} Q_{i-1}-Q_{i-2}.\label{piqi}
\end{equation}
\begin{theorem}[{\cite[Proposition 10.2.2]{CLS}}]\label{HJ}
With the above notation, we have:
\begin{itemize}
\item [\rm(i)] $P_0 < P_1 < \dots < P_{r+1},\; Q_0 < Q_1 < \dots < Q_{r+1}$;
\item [\rm(ii)] $P_{i-1} Q_{i} -P_{i} Q_{i-1}=1$ for any $1 \leq i \leq r+1$;
\item [\rm(iii)] $\displaystyle\frac{b}{t}=\frac{P_{r+1}}{Q_{r+1}} < \frac{P_r}{Q_r}\dots < \frac{P_2}{Q_2}$.
\end{itemize}
\end{theorem}
Let 
\[
\rho_i:=-P_i\mbox{\bf u}_1+\{(\alpha+1)P_i-Q_i\}\mbox{\bf u}_2 \quad (0 \leq i \leq r+1),
\]
and 
\[
\mathcal C_i:=\mathbb Q_{\geq 0}\rho_{i-1}+\mathbb Q_{\geq 0}\rho_i \quad (1 \leq i \leq r+1).
\]
We denote by $\widetilde{\E'}$ the toroidal spherical $SL(2) \times \C^*$-variety whose colored fan has $(\mathcal C_1, \phi),\; \dots,\; (\mathcal C_{r+1}, \phi)$ as its maximal colored cones. Then, 
\[
\xymatrix{
\widetilde{\E'} \ar[r] & \E'
}
\]
is the minimal resolution. 
The main result of this article is: 
\begin{theorem}\label{main theorem}
The main component $\mathcal H^{main}$ is isomorphic to $\widetilde{\E'}$. 
\end{theorem}
\section{First step towards the proof of Theorem \ref{main theorem}}\label{s-proof}
In this section, we construct an equivariant morphism $\widetilde{\E'} \to \P(V^{\vee})$ defined by a base-point-free $V \subset \Gamma(\widetilde{\E'}, \mathcal O(\delta))$, where $\delta$ is an $SL(2) \times \C^*$-stable Cartier divisor, and show that the natural morphism $\Phi : \widetilde{\E'} \to \E \times \P(V^{\vee})$ is a closed immersion (Proposition \ref{phicl}).

Let $D_i$ be an $SL(2) \times \C^*$-stable prime divisor on $\widetilde{\E'}$ corresponding to the extremal ray $\mathbb Q_{\geq 0} \rho_i$. 
Then, with the notation defined in \S \ref{s-spherical}, we have
\begin{equation*}
\mathcal D(\widetilde{\E'})=\{D_0,\; \dots,\; D_{r+1},\; \widetilde{S^+},\; \widetilde{S^-}\},
\end{equation*}
where 
$\widetilde{S^+}$ (resp. $\widetilde{S^-}$) is a non-$SL(2) \times \C^*$-stable prime divisor on $\widetilde{\E'}$ such that its image under the canonical birational morphism $\widetilde{\E'} \to \E$ is the $\tilde{B}$-stable divisor $S^+$ (resp. $S^-$) on $\E$. 
By definition, we have 
\[
v_{D_i}(f)=\rho_{v_{D_i}}(\chi_f)=\rho_i(\chi_f)
\]
for any $f \in \C(\mathfrak U)^{\tilde{B}}$. 
We remark that we have $\rho_0=\rho_{v_{D'}}$ and  $\rho_{r+1}=\rho_{v_{D}}$. 

For each $0 \leq i \leq r+1$, we define 
\[
e_i:=(\alpha +1 + m)P_i - Q_i, \quad l_i:=(\alpha +1)P_i - Q_i, \quad n_i:=-pe_i+ql_i.
\]
The next lemma is a consequence of Theorem \ref{HJ}. 
\begin{lemma}\label{ni}
We have:
\begin{itemize}
\item [\rm(i)] $n_i=k(tP_i-bQ_i)$ for any $0 \leq i \leq r+1$; 
\item [\rm(ii)] $n_i=c_{i-1}n_{i-1}-n_{i-2}$ for any  $2 \leq i \leq r+1$;
\item [\rm(iii)] $n_0=q-p > n_1=q-p-\beta > n_2 > \dots > n_{r-1} > n_r=k > n_{r+1}=0$.
\end{itemize}
\end{lemma}

For each $0 \leq i \leq r+1$, set
\begin{equation*}
\sigma_i:=Z^{e_i} W^{l_i}, \quad
f_i:=\prod_{0 \leq j \leq i} \sigma_j. 
\end{equation*}
\begin{lemma}\label{gg}
With the preceding notation, the following properties are true. 
\begin{itemize}
\item [\rm(i)] Let $0 \leq i,\; j \leq r+1$. Then we have 
\[
v_{D_j}(\sigma_i)
\begin{cases}
 > 0 \quad (\mbox{if}\; \; i >j)\\
 = 0 \quad (\mbox{if}\; \; i=j)\\
 < 0 \quad (\mbox{if}\; \; i <j).
\end{cases}
\]
In particular, we have $v_{D_j}(\sigma_{j+1})=1$ and 
$v_{D_j}(\sigma_{j-1})=-1$. 
\item [\rm(ii)] We have $v_{D_i}(f_i)=v_{D_i}(f_{i-1})$.
\end{itemize}
\end{lemma}
\begin{proof}
Since $v_{D_j}(\sigma_i)=\rho_j(\chi_{\sigma_i})=P_jQ_i-P_iQ_j$, 
we get (i) by Theorem \ref{HJ}. 
Item (ii) follows from the definition of $f_i$ and (i). 
\end{proof}
Let $\widetilde{E_i}$ be the simple spherical open subvariety of $\widetilde{\E'}$ corresponding to the colored cone $(\mathcal C_i,\; \phi)$, and $Y_i$ the unique closed orbit in $\widetilde{E_i}$. Then we have 
\begin{equation*}
\mathcal D(\widetilde{E_i})=\{{D_{i-1}}{|_{\widetilde{E_i}}},\; {D_i}{|_{\widetilde{E_i}}},\; \widetilde{S^+}{|_{\widetilde{E_i}}},\; \widetilde{S^-}{|_{\widetilde{E_i}}}\}, \quad 
\mathcal D_{Y_i}(\widetilde{E_i})=\{{D_{i-1}}{|_{\widetilde{E_i}}},\; {D_i}{|_{\widetilde{E_i}}}\}.
\end{equation*}
Let us consider the following $SL(2) \times \C^*$-stable  divisor on $\widetilde{\E'}$:
\begin{equation*}
\delta:=\sum_{1 \leq i \leq r+1} v_{D_i}(f_i^{-1}) D_i.
\end{equation*}
Though the Cartierness of $\delta$ follows immediately from the smoothness of $\widetilde{\E'}$, we check the criterion for a Weil divisor to be Cartier given in Theorem \ref{cartiercriterion} as a preparation for the proof of Lemma \ref{ggc}: 
with the notation used in Theorem \ref{cartiercriterion}, we see by Lemma \ref{gg} (ii) that $f_{Y_i}=f_{i-1}^{-1}$ satisfies the required condition. 
\begin{lemma}\label{ggc}
The Cartier divisor $\delta$ is generated by global sections.
\end{lemma}
\begin{proof}
Taking Theorem \ref{criterion for gg} into account,  it is enough to show the following: 
\begin{itemize}
\item [\rm(a)] 
$v_{D_j}(f_{i-1}^{-1}) \leq v_{D_j}(f_{j-1}^{-1})$ and $v_{D_{j-1}}(f_{i-1}^{-1})\leq v_{D_{j-1}}(f_{j-1}^{-1})$ hold for any $1 \leq i ,\;  j \leq r+1$; and 
\item [\rm(b)] $v_{\widetilde{S^+}}(f_{i-1}^{-1})\leq 0$ and $v_{\widetilde{S^-}}(f_{i-1}^{-1}) \leq 0$ hold for any $1 \leq i \leq r+1$. 
\end{itemize}
Condition (a) follows from Lemma \ref{gg}. 
By a direct calculation, we have $v_{\widetilde{S^+}}(f_{i-1})=\sum_{0 \leq j \leq i-1} e_j$ and $v_{\widetilde{S^-}}(f_{i-1})=\sum_{0 \leq j \leq i-1} l_j$. This shows (b).  
\end{proof}
\begin{remark}\label{linearization}
Since $\delta$ is $SL(2) \times \C^*$-stable, there is a linearization of the action of $SL(2) \times \C^*$ with respect to the 
line bundle $\mathcal O(\delta)$ such that the induced action on $\Gamma(\widetilde{\E'},\; \mathcal O(\delta))$ coincides with that 
on the function field $\C(\widetilde{\E'})$ (see \cite{IUJA}). 
\end{remark}
Let 
\[
V:=\langle(SL(2) \times \C^*) \cdot f_i\; :\; 1 \leq i\leq r\rangle,
\]
which is isomorphic to 
$\bigoplus_{1 \leq i \leq r} V(e_0+e_1+ \dots + e_i) \otimes V(l_0+l_1+ \dots + l_i)$. Here, $V(n)$
stands for the irreducible $SL(2)$-representation of highest weight $n$. We can take 
\[
\mathcal A:=\left\{X^{e_0+e_1+ \dots + e_i-e}Z^eY^{l_0+l_1+ \dots + l_i-l}W^l\; :\;
\begin{matrix}
1 \leq i \leq r;\qquad  \qquad \qquad \; \;  \\
0 \leq e \leq e_0+ e_1 +\dots + e_i;\; \\
 0 \leq l \leq l_0+ l_1+\dots + l_i \quad \; \;  
\end{matrix}
\right\}
\]
as a basis of $V$. 
\begin{lemma}
The vector space $V$ is an $SL(2) \times \C^*$-submodule of $\Gamma(\widetilde{\E'}, \mathcal O(\delta))$. 
\end{lemma}
\begin{proof}
We show that $f_i \in \Gamma(\widetilde{\E'}, \mathcal O(\delta))$ holds for every $1 \leq i \leq r$.
For any $1 \leq j \leq r+1$, we have 
\[
\Div(f_i){|_{\widetilde{E_j}}}=v_{\widetilde{S^+}}(f_i)
{\widetilde{S^+}}{|_{\widetilde{E_j}}}+v_{\widetilde{S^-}}(f_i){\widetilde{S^-}}{|_{\widetilde{E_j}}} 
+v_{D_{j-1}}(f_i) {D_{j-1}}{|_{\widetilde{E_j}}}+ v_{D_j}(f_i){D_j}{|_{\widetilde{E_j}}}
\]
and 
\[
\delta{|_{\widetilde{E_j}}}=v_{D_{j-1}}(f_{j-1}^{-1}) {D_{j-1}}{|_{\widetilde{E_j}}}+v_{D_j}(f_j^{-1}) {D_j}{|_{\widetilde{E_j}}}.
\]
Thus we get ${\Div(f_i)}{|_{\widetilde{E_j}}}+\delta{|_{\widetilde{E_j}}} \geq 0$ by comparing each coefficient using 
the condition (a) in the proof of Lemma \ref{ggc}.
\end{proof}
Therefore, we obtain a natural equivariant morphism 
\[
\xymatrix{
\Phi : \widetilde{\E'} \ar[r] &  E_{l,m} \times \P(V^{\vee}).
}
\]
We show that $\Phi$ is a closed immersion.  
Recall that $\widetilde{\E'}$ is covered by simple open subembeddings $\widetilde{E_1},\; \dots,\; \widetilde{E_{r+1}}$ and that $\widetilde{E_i}=(SL(2) \times \C^*) {(\widetilde{E_i})}_0$, where
\begin{equation*}
(\widetilde{E_i})_0=\widetilde{E_i} \setminus \bigcup_{D \in \mathcal D(\widetilde{E_i}) \setminus \mathcal D_{Y_i}(\widetilde{E_i})} D=\widetilde{E_i} \setminus ({\widetilde{S^+}}{|_{\widetilde{E_i}}} \cup {\widetilde{S^-}}{|_{\widetilde{E_i}}})
\end{equation*}
following the notation defined in \S \ref{s-spherical}. 
Also, we have 
$(\widetilde{E_i})_1=\mathfrak U \cap \{ZW \neq 0\}$, $(\E)_0=\E$, and $(\E)_1=\mathfrak U$. 
Therefore, it follows from Remark \ref{coordinate ring} that 
\[
\C[(\widetilde{E_i})_0]=\left\{F \in \C[\mathfrak U]_{ZW}\; :\; v_{D_{i-1}}(F) \geq 0,\; v_{D_i}(F) \geq 0\right\}
\]
and
\[
\C[E_{l,m}]=\left\{F \in \C[\mathfrak U]\; :\; v_{D_{r+1}}(F) \geq 0\right\}.
\]
Let $L $ be a subring of $\C(\mathfrak U)$ defined by 
\[
L=\{F \in \C[\mathfrak U]_{ZW}\; :\; v_{D_{r+1}}(F) \geq 0\},
\]
and consider an open subset 
\begin{equation*}
U_i:=\Spec \left(L\left[\frac{f^{\vee}}{f_{i-1}^{\vee}}\; :\; f \in \mathcal A\right]\right)\qquad (1 \leq i \leq r+1)
\end{equation*}
of $E_{l,m} \times \mathbb P(V^{\vee})$, where $f^{\vee}$ denotes the dual basis of $f$. 
Also, consider a homomorphism 
\[
\xymatrix{
\Phi^{\#}_i : L\left[\frac{f^{\vee}}{f^{\vee}_{i-1}}\; :\; f \in \mathcal A\right] \ar[r] & \C[(\widetilde{E_i})_0]
}
\]
defined by sending $F  \frac{f^{\vee}}{f^{\vee}_{i-1}}$, where $F \in L$, to 
$F \frac{f}{f_{i-1}}$. 
\begin{lemma}\label{well-defined}
The homomorphism $\Phi^{\#}_i$ is well-defined. \end{lemma}
\begin{proof}
Let $F=\frac{X^{e'}Y^{l'}Z^eW^l}{(ZW)^d} \in L$, where $e,e', l,l', d \geq 0$. 
Since $F$ is invariant under the action of $G_m$, we have $e'+e-l'-l=mc$ for some $c \in \mathbb Z$. 
Therefore, for any $0 \leq j \leq r+1$, we have 
\[
v_{D_j}(F)=-P_j(l'+l-d)+\{(\alpha+1)P_j-Q_j\}c.
\] 
Taking $j=r+1$, we get 
$\frac{pmc}{q-p}=\frac{apc}{b} \geq l'+l-d$ 
by using the equations \eqref{akb}, \eqref{mp}, and \eqref{t}. 
This implies that $c \geq 0$. 
Therefore, we have 
\[
v_{D_j}(F)\geq -P_j \frac{pmc}{q-p}+\{(\alpha+1)P_j-Q_j\}c=\frac{n_j c}{q-p} \geq 0
\]
concerning Lemma \ref{ni}. 
Thus, $L \subset \C[(\widetilde{E_i})_0]$. 
Also, we have  $f_j/f_{i-1} \in \C[(\widetilde{E_i})_0]$ by the condition (a) in the proof of Lemma \ref{ggc}, and hence $f/f_{i-1} \in \C[(\widetilde{E_i})_0]$ for any $f \in \mathcal A$. 
\end{proof}
\begin{lemma}\label{cl}
The homomorphism $\Phi^{\#}_i$ is surjective.
\end{lemma}
\begin{proof}
Let $F=\frac{X^{e'}Y^{l'}Z^eW^l}{(ZW)^d} \in \C[(\widetilde{E_i})_0]$, where $e, e', l, l', d \geq 0$. 
As in the proof of Lemma \ref{well-defined}, we can write $e'+e=l'+l+mc$ for some $c \in \mathbb Z$. 
By a direct calculation, we see that $v_{D_{r+1}}(F)\geq 0$ if and only if $\frac{apc}{b} \geq l'+l-d$. 
In the following, we assume that 
$v_{D_{r+1}}(F) < 0$, since otherwise we have $F=\Phi^{\#}_i(F)$. 
Set $F'=F/\sigma_i$. Then, as an element of the function filed $\C(\mathfrak U)$, we can write $F$ as 
\begin{equation}
F= 
F' \frac{f_i}{f_{i-1}} \label{F}.
\end{equation}
We claim that the following two conditions hold: 
(a) $v_{D_{r+1}}(F')> v_{D_{r+1}}(F)$; 
(b) $F' \in \C[(\widetilde{E_i})_0]$. 
Indeed, the condition (a) follows from Lemma \ref{gg}.  
Also, in view of Lemma \ref{gg}, it suffices to show that $v_{D_{i-1}}(F) \geq 1$  holds to get (b). 
Suppose otherwise, i.e., 
$v_{D_{i-1}}(F)=0$. Then we have 
\begin{equation}
(l'+l-d)P_{i-1}=c\left\{(\alpha+1)P_{i-1}-Q_{i-1}\right\}. \label{di-1}
\end{equation}
If $i=1$, then the conditions $v_{D_{0}}(F)=0$ and $v_{D_1}(F)\geq 0$ imply that $c=0$ and $0 \geq l'+l-d$, which contradicts to our assumption that $v_{D_{r+1}}(F) <0$.
Let $i \geq 2$. 
By  \eqref{piqi}, we have 
$0 \leq v_{D_{i}}(F)=c_{i-1}v_{D_{i-1}}(F)-v_{D_{i-2}}(F)=- v_{D_{i-2}}(F)$,  
and hence
\begin{equation}
c\left\{(\alpha+1)P_{i-2}-Q_{i-2}\right\} \leq (l'+l-d)P_{i-2}. \label{di-2}
\end{equation}
If $i=2$, then we can show in a similar way that the hypothesis $v_{D_1}(F)=0$ leads to a contradiction. 
If $i>2$, then by \eqref{di-1} and \eqref{di-2} we have 
$c \frac{{Q_{i-1}}}{{P_{i-1}}} \leq c \frac{{Q_{i-2}}}{{P_{i-2}}}$, 
and thus $c=0$ concerning  
Theorem \ref{HJ}. 
In a same manner, we see that this contradicts to the assumption. 
Therefore, we have  $F' \in \C[(\widetilde{E_i})_0]$.  
The conditions 
(a) and (b) and the equation \eqref{F} yield that there is an $F'' \in \C[(\widetilde{E_i})_0]$ with $v_{D_{r+1}}(F'')\geq 0$ such that 
$F=F'' \left(\frac{{f_i}}{{f_{i-1}}}\right)^t$ holds 
for some $t > 0$. 
Thus, we get 
$F=\Phi^{\#}_i\left(F'' \left(\frac{{f_i^{\vee}}}{{{f_{i-1}}^{\vee}}}\right)^t\right)$. 
\end{proof}
As a consequence of Lemmas \ref{well-defined} and \ref{cl}, we obtain: 
\begin{proposition}\label{phicl}
The morphism $\Phi: \widetilde{\E'} \to \E \times \P(V^{\vee})$ is a closed immersion. 
\end{proposition}
\section{Generators as a module over the invariant ring}\label{s-gen} 
For each $n \geq 0$, consider the following irreducible $SL(2)$-representations:
\[
A(n):=\Sym^n\langle X_1, X_2\rangle \cong V(n), \quad 
B(n):=\Sym^n\langle X_3, X_4\rangle\cong V(n). 
\]
Also, define $C(n):=\langle X_0^n\rangle \cong V(0)$ for each $n \in \mathbb Z$, and set 
\begin{equation*}
F_{n_0,0}:=A(e_0) \otimes B(l_0),\quad 
F_{n_i, 0}:=A(e_i) \otimes B(l_i) \oplus C(n_i) \quad (1 \leq i \leq r).
\end{equation*}
The goal of this section is to prove the following 
\begin{theorem}\label{generator2}
For any $0 \leq i \leq r$, the weight space $S_{(n_i,0)}$ is generated by $F_{n_i,0}$ as a module over the invariant ring $S^{\G}$. 
\end{theorem}
We prepare notations and lemmas that we need for the proof of Theorem \ref{generator2}. Some of them have already appeared in \cite[\S 4]{Ku}. 

Let $R:=\C[X_0, X_1, X_3] \subset A$. The polynomial ring $R$ has  a natural $\Z \times \Z / m\Z$-grading defined by the $\G$-action: $R=\bigoplus_{(n,d) \in \Z \times \Z/m\Z} R_{(n,d)}$. 
Concerning that 
$X_1$ and $X_2$ (resp. $X_3$ and $X_4$) have the same $SL(2) \times \C^* \times \G$-weight, it suffices to determine a subspace of $R_{(n_i,0)}$ that generates $R_{(n_i,0)}$ over the invariant ring $R^{\G}$ in proving Theorem \ref{generator2}. For each $c,n \in \mathbb Z$, we consider the vector subspaces 
\[
R^c:=\langle X_0^{d_0}X_1^{d_1}X_{3}^{d_{3}} \in R\; :\; d_1-d_{3}=c \rangle
\]
and
\[
R_n:=\langle X_0^{d_0}X_1^{d_1}X_{3}^{d_{3}} \in R\; :\; d_0-pd_1+qd_{3}=n \rangle
\] 
of $R$. 
Then we have 
\begin{equation*}
R=\bigoplus_{c \in \mathbb Z} R^c=\bigoplus_{n \in \mathbb Z} R_n.
\end{equation*}
Let $R^c_n:=R^c \cap R_n$. Then, 
the weight space $R_{(n,d)}$ is described as follows: 
\begin{equation*}
R_{(n,d)}=\bigoplus_{c \equiv d\; (mod\; m)} R^c_n. 
\end{equation*}
\begin{remark}\label{invariant ring}
By the proof of {\cite[Theorem 1.6]{BH}}, we see that the invariant ring $R^{\G}=R_{(0,0)}$ is described as follows:  
\[
R^{\G}=\C[X_0^{pu_1-qu_2}X_{1}^{u_1}X_{3}^{u_2}\; :\; (u_1,u_2) \in M^+_{l,m}].
\]
\end{remark}
\begin{example}\label{fex}
Let $l=p/q=1/4$, and $m=2$. 
By using an algorithm described in \cite{Pa} for finding a system of generators of the semigroup $M^+_{l,m}$, we see that $M^+_{\frac{1}{4}, 2}$ is minimally generated by $(2,0),\; (5,1)$, and $(8,2)$. 
Therefore, 
\[
R^{\G}=\C[X_0^2X_1^2,\; X_0X_1^5X_3,\; X_1^8X_3^2]. 
\]
\end{example}
\begin{lemma}[{\cite[Lemma 4.6]{Ku}}]\label{minc}
For any $(n,d) \in \mathbb Z \times \mathbb Z/m \mathbb Z$, the minimum
\[
c_{(n,d)}:=\min\{c \in \mathbb Z\; :\; c \equiv d\; (mod\; m),\; R^c_n \neq 0\}
\]
exists.
\end{lemma}
\begin{example}[{\cite[Example 4.7]{Ku}}]\label{positive}
If $0 \leq n \leq q-p$, then $c_{(n,0)}=0$. 
We have  $R^0_n= \langle X_0^n \rangle$ if $0 \leq n < q-p$, and 
$R^0_{q-p}=\langle X_0^{q-p},\; X_1X_{3} \rangle$. 
\end{example}
We define another grading on $R$ such that each graded component is finite-dimensional, which makes it easier to analyze the structure of the weight space $R_{(n,d)}$. 
For that purpose, consider a $\mathbb Z$-linear map $\mu : \mathbb Z^3 \to \mathbb Z^3$ defined by
\[
(d_0, d_1, d_{3}) \mapsto \mu(d_0, d_1, d_{3}):=(d_0-pd_1+qd_{3},\; d_1-d_{3},\; pd_1-qd_{3}).
\]
We see that $\mu$ is injective. 
Let us denote by $\Lambda$ the image of $\mu|_{\mathbb Z_{\geq 0}^3}$, and define 
\begin{equation*}
R_{\lambda}:=\langle X_0^{d_0}X_1^{d_1}X_{3}^{d_{3}} \in R\; :\; \mu(d_0, d_1, d_{3})=\lambda \rangle
\end{equation*}
for each $\lambda \in \Lambda$. 
Then we have
\begin{equation*}
R=\bigoplus_{\lambda \in \Lambda} R_{\lambda}. 
\end{equation*}
Next, consider the projection $\tilde{\mu} : \mathbb Z^3 \to \mathbb Z^2,\; (n,c,\w) \mapsto (n,c)$ to the first and the second factor. 
Set $\mu':=\tilde{\mu} \circ \mu$, and 
denote by $\Lambda'$ the image of $\mu'|_{\mathbb Z_{\geq 0}^3}$. Then we have
\begin{equation*}
R=\bigoplus_{(n,c) \in \Lambda'} R^c_n, \quad 
R^c_n=\bigoplus_{\lambda \in \tilde{\mu}^{-1}(n,c)} R_{\lambda}. \label{cn}
\end{equation*}
\begin{lemma}[{\cite[Lemmas 4.8 and 4.10]{Ku}}]\label{1dim}
With the preceding notation, the following properties hold. 
\begin{itemize}
\item [\rm(i)] Let $\lambda=(n,c,\w) \in \Lambda$. Then, the vector space $R_{\lambda}$ is spanned by 
\begin{equation*}
f_{\lambda}:=X_0^{n+\w}X_1^{\frac{qc-\w}{q-p}}X_{3}^{\frac{pc -\w}{q-p}}. \label{function}
\end{equation*}
\item [\rm(ii)] For any $\lambda,\; \lambda' \in \Lambda$, we have 
$f_{\lambda} f_{\lambda'}=f_{\lambda+\lambda'}$. 
\item [\rm(iii)] Let $(n,c) \in \Lambda'$. Then we have $\w-\w' \in (q-p)\mathbb Z$ for any $(n,c,\w),\; (n,c,\w') \in \tilde{\mu}^{-1}(n,c)$.  
\end{itemize}
\end{lemma}
We see that 
\[
\w_{(n,c)}^{\max}:=\max\{\w \in \mathbb Z\; :\; (n,c,\w) \in \tilde{\mu}^{-1}(n,c)\}\]
and 
\[
\w_{(n,c)}:=\min\{\w \in \mathbb Z\; :\; (n,c,\w) \in \tilde{\mu}^{-1}(n,c)\}
\]
exist for any $(n,c) \in \Lambda'$, and that  the vector space $R^c_n$ is finite-dimensional.
\begin{lemma}\label{max}
Let $(n,c) \in \Lambda'$. If $c < 0$, then $\w_{(n,c)}^{\max}=qc$. Otherwise, $\w_{(n,c)}^{\max}=pc$. 
\end{lemma}
\begin{proof}
Let $\mu(d_0, d_1, d_{3})=(n,c,\w_{(n,c)}^{\max})$. 
We claim that either $d_1=0$ or $d_{3}=0$ holds. 
Indeed, if $d_1> 0$ and $d_3 >0$, then we have 
\[
\mu(d_0+q-p, d_1-1, d_{3}-1)=(n,c,\w_{(n,c)}^{\max}+q-p) \in \tilde{\mu}^{-1}(n,c),
\]
which contradicts to the maximality of $\w_{(n,c)}^{\max}$. 
Thus, if $c<0$ then we see that $d_1=0$, and therefore $\w^{\max}_{(n,c)}=qc$. 
\end{proof}
\begin{lemma}[{\cite[Lemma 4.11]{Ku}}]\label{min}
Let $(n,c,\w) \in \Lambda$. Then, 
we have $n+\w < q-p$ if and only if $\w=\w_{(n,c)}$. 
\end{lemma}
\begin{corollary}\label{remainder1}
Let $(n,c),\; (n', c') \in \Lambda'$. Then the following properties are true. 
\begin{itemize}
\item [\rm(i)] If $n=0$, then $0 \leq \w_{(0,c)} < q-p$.
\item [\rm(ii)]  
We have $\w_{(n+n', c+c')}=\w_{(n,c)}+\w_{(n', c')}$ if and only if $\w_{(n,c)}+\w_{(n', c')}+n+n' <q-p$. 
\end{itemize}
\end{corollary}
\begin{proof}
First, we have $\w_{(0,c)} <q-p$ by Lemma \ref{min}. Let $(d_0, d_1,d_3) \in \mathbb Z^3$ be such that $\mu(d_0, d_1, d_{3})=(n,c,\w_{(n,c)})$. 
Then we have $d_0=n+\w_{(n,c)}$ by Lemma \ref{1dim} (i), and therefore we have $\w_{(0,c)} =d_0 \geq 0$ if $n=0$. 
Item (ii) follows from the fact that $(n+n', c+c', \w_{(n,c)}+\w_{(n', c')}) \in \tilde{\mu}^{-1}(n+n', c+c')$ and Lemma \ref{min}. 
\end{proof}
\begin{example}\label{zi}
We have $\w^{\max}_{(n_i,mP_i)}=pmP_i$ by Lemma \ref{max}. By a direct calculation, we obtain the following: 
\[
\w^{\max}_{(n_i,mP_i)}+n_i=\{\alpha(q-p)+\beta\}P_i+(q-p-\beta)P_i-(q-p)Q_i=\{(\alpha+1)P_i-Q_i\}(q-p).
\] 
Therefore, we see that $(n_i, mP_i, -n_i) \in \tilde{u}^{-1}(n_i, mP_i)$. It follows that $\w_{(n_i, mP_i)}=-n_i$, since $n_i+(-n_i)<q-p$. 
Also, we can calculate that $f_{(n_i, mP_i, \w_{(n_i, mP_i)})}=X_1^{e_i}X_{3}^{l_i}$ by using Lemma \ref{1dim} (i).  
\end{example}
\begin{definition}
For any positive integers $m_1$ and $m_2$, we denote by 
$\Rem[m_1, m_2]$ the remainder of $m_1$ divided by $m_2$. 
\end{definition}
\begin{corollary}\label{preremainder}
Let $(n,c) \in \Lambda'$, and suppose that $n\geq 0$ and that $c>0$.  
Then, we have $\w_{(n,c)} \geq 0$ if and only if $\Rem[pc, q-p]+n < q-p$. 
\end{corollary}
\begin{proof}
By Lemmas \ref{1dim} (iii) and \ref{max}, we have $pc=x(q-p)+\w_{(n,c)}$ for some $x  \geq 0$. 
If $\w_{(n,c)} \geq 0$, then we get 
$\Rem[pc, q-p]+n=\w_{(n,c)}+n <q-p$ by Lemma \ref{min}. 
Otherwise, we have $
\Rem[pc, q-p]=x'(q-p)+\w_{(n,c)}$ for some $x' > 0$. 
Therefore, $\Rem[pc, q-p]+n=x'(q-p)+\w_{(n,c)}+n \geq q-p$,
since $\w_{(n,c)}+n \geq 0$ concerning Lemma \ref{1dim} (i). 
\end{proof}
\begin{definition}[{\cite[Definition 4.12]{Ku}}]
For each $(n,d) \in \mathbb Z \times \mathbb Z/m \mathbb Z$, we define:
\begin{itemize}
\item [\rm(i)] $\Lambda_{(n,d)}:=\{(n,c,\w) \in \Lambda \; :\; c \equiv d\; (mod\; m)\}$;
\item [\rm(ii)] $\lambda_{(n,d)}:=(n,c_{(n,d)}, \w_{(n,c_{(n,d)})}) \in \Lambda_{(n,d)}$.
\end{itemize}
\end{definition}
Using the notation defined above, we obtain different ways of expressing the weight space $R_{(n,d)}$: 
\begin{equation*}
R_{(n,d)}
=\bigoplus_{\substack{
c \equiv d\; (mod\; m) \\
c \geq c_{(n,d)}}} R^c_n 
 =\bigoplus_{\substack{
c \equiv d\; (mod\; m)\\
c \geq c_{(n,d)}}}\left(\bigoplus_{\lambda \in \tilde{\mu}^{-1}(n,c)} R_{\lambda}\right) 
 =\bigoplus_{\lambda \in \Lambda_{(n,d)}} R_{\lambda}. \label{relation1}
 \end{equation*}

Now since we have $R_{(n_i,0,\w_{(n_i,0)})}=\langle X_0^{n_i} \rangle$ and $R_{(n_i,mP_i,\w_{(n_i,mP_i)})}=\langle X_1^{e_i}X_3^{l_i} \rangle$, we see that Theorem \ref{generator2} follows as a consequence of the next
 \begin{proposition}\label{long}
For any $0 \leq i \leq r$, 
the weight space $R_{(n_i, 0)}$ is generated by $R^0_{n_i}$ and $R^{mP_i}_{n_i}$ as a module over the invariant ring $R^{\G}$.
\end{proposition}
The rest of this section is devoted mostly to the proof of Proposition \ref{long}. 
Recall that we have considered the Hirzebruch--Jung continued fraction of $b/t$ in \S \ref{s-flat}. 
Set $t_1:=t$. Then we have the following equations that arise from the \emph{modified Euclidean algorithm} (see \cite[\S 10]{CLS} for more details):
\begin{align}
& b=c_1t_1-t_2, \notag \\
& t_1=c_2t_2-t_3, \notag \\
& \dots \label{relation}\\
& t_{i-1}=c_it_i-t_{i+1},\notag \\
& \dots \notag \\
& t_{r-1}=c_rt_r. \notag
\end{align}
We can easily see that the following equation holds for any $2 \leq i \leq r$:
\begin{equation}
b-t_1=(c_1-2)t_1+(c_2-2)t_2+ \dots + (c_{i-1}-2)t_{i-1}+t_{i-1}-t_i. \label{b}
\end{equation}
Since $b=n_0/k$ and $t_1=t=n_1/k$,  Lemma \ref{ni} and \eqref{relation} yield that 
$t_i=n_i/k$ holds for any $1 \leq i \leq r$. 

Now, let us consider the following two conditions: 
\begin{itemize}
\item [\rm{(C1)}] $1 < i \leq r+1$;
\item [\rm{(C2)}] $1 \leq \exists l \leq i-1$ such that $c_l >2$, $c_{l+1}=\dots=c_{i-1}=2$. 
\end{itemize}
Assume that conditions (C1) and (C2) hold, and let $x$ be any integer such that $0 \leq x < P_i-P_{i-1}$. 
The quotient of $t_1(P_{i-1}+x)$ divided by $b$ is always not less that $Q_{i-1}$, since we have $t_1P_{i-1}=bQ_{i-1}+t_{i-1}$ with $0 \leq t_{i-1} <b$ by Lemma \ref{ni}. Keeping this in mind,   
let $Q_{i-1}+\theta_x$ (resp. $\Theta[x]$) be the quotient (resp. the remainder) of $t_1(P_{i-1}+x)$ divided by $b$, namely
\[
t_1(P_{i-1}+x)=b(Q_{i-1}+\theta_x)+\Theta[x], \quad 
\Theta[x]=\Rem[t_1(P_{i-1}+x), b].
\]
Then we have $\Theta[x]=t_{i-1}+t_1x-b\theta_x$. 
\begin{remark} \label{e}
With the above notation and assumption, we have the following. 
\begin{itemize}
\item [\rm(i)] Since $t_1 < b$, we see that $\theta_x-\theta_{x-1} \in \{0,1\}$. Furthermore, the following properties are true.
\begin{itemize}
\item [$\bullet$] We have $\theta_x-\theta_{x-1}=0$ if and only if $\Theta[{x-1}]+t_1-b < 0$. In this case, $\Theta[x]=\Theta[{x-1}]+t_1$. 
\item [$\bullet$] We have $\theta_x-\theta_{x-1}=1$ if and only if $\Theta[x-1]+t_1-b \geq 0$. In this case, $\Theta[x]=\Theta[x-1]+t_1-b$. 
\end{itemize}
\item [\rm(ii)] Since $P_i-P_{i-1} <b$, we have $\Theta[x]=\Theta[{x'}]$ if and only if $x=x'$. 
\item [\rm(iii)] We have $\Theta[0]=t_{i-1}$. 
\item [\rm(iv)] By the assumption (C2), we have $P_i-P_{i-1}=(c_1-1)P_1+(c_2-2)P_2+\dots + (c_{l}-2)P_{l}=P_{l+1}-P_{l}$. 
\end{itemize}
\end{remark}
\begin{lemma}\label{generalxj}
Assume that the conditions (C1) and (C2) hold. 
If $\Theta[x]=t_{i-1}+(c_1-2)t_1+ \dots +(c_{j-1}-2)t_{j-1}+(c_{j}-1)t_{j}$ 
holds for some $1 \leq j \leq l$, then we have $\Theta[{x+1}]=t_{i-1}+t_{j+1}$. 
\end{lemma}
\begin{proof}
By a direct calculation using \eqref{b}, we have $\Theta[x]+t_1-b=t_{i-1}+t_{j+1} > 0$, and therefore $\Theta[{x+1}]=\Theta[x]+t_1-b=t_{i-1}+t_{j+1}$ by Remark \ref{e} (i). 
\end{proof}
The next lemma is the key in proving Proposition \ref{long}. 
\begin{lemma}\label{keyprop}
Assume that the conditions (C1) and (C2) hold. 
Then, the following properties are true for any $1 \leq j < l$. 
\begin{itemize}
\item [\rm(i)] Let $P_j \leq x < P_{j+1}$, and denote by $\kappa$ (resp. $\varepsilon$) the quotient (resp. the remainder) of $x$ divided by $P_j$, i.e., $x=\kappa P_j+\varepsilon$. Then, we have $\Theta[x]=\Theta[\varepsilon]+\kappa t_j$. 
In particular, we have $\Theta[x] > t_{i-1}$. 
\item [\rm(ii)] Let $M_j:=\max\{\Theta[x]\; :\; 0 \leq x < P_{j+1}\}$. 
Then, 
$M_j=\Theta[{P_{j+1}-P_j}]=t_{i-1}+b-t_j+t_{j+1}$.
\end{itemize}
\end{lemma}
\begin{proof}
First we remark that the following holds: 
\[
t_{i-1}+b-t_j+t_{j+1}=t_{i-1}+(c_1-1)t_1+(c_2-2)t_2+ \dots+(c_j-2)t_j. 
\]

We proceed by induction on $j$. 
Suppose that $j=1$. 
Firstly, we have  
\begin{align*}
\Theta[0]+t_1-b
& =t_{i-1}+t_1-b \\
& =t_i-\{(c_1-2)t_1+ \dots + (c_l-2)t_l+(c_{l+1}-2)t_{l+1}+\dots +(c_{i-1}-2)t_{i-1}\} \\ 
& =t_i-\{(c_1-2)t_1+ \dots + (c_{l}-2)t_{l}\} \leq t_i-t_{l} <0 .
\end{align*}
Therefore, we have $\Theta[{P_1}]=\Theta[1]=t_{i-1}+t_1$ by Remark \ref{e}.  
 In the same way, we see that 
$\Theta[{x-1}]+t_1-b \leq t_i-t_{l} <0$ holds for any $P_1 < x<P_2$, and thus we get $\Theta[x]=t_{i-1}+xt_1$. 
Further, this yields that $M_1=\Theta[P_2-P_1]$. 

Next, suppose that $j >1$. 
Notice that we have $P_{j+1}=(c_j-1)P_j+(P_j-P_{j-1})$.  
In the following, we divide the proof into three steps. 

\emph{Step 1.} We show by induction on $\kappa$ that 
\[
\Theta[{\kappa P_j}]=t_{i-1}+\kappa t_j
\]  
holds for any $1 \leq \kappa \leq c_j-1$. 
Let $\kappa=1$. 
By the induction hypothesis for (i), we see that
\[
\Theta[P_j-1]=\Theta[P_{j-1}-P_{j-2}-1]+(c_{j-1}-1)t_{j-1}.
\]
Taking Remark \ref{e} into account, either 
\[
\Theta[P_{j-1}-P_{j-2}]=\Theta[P_{j-1}-P_{j-2}-1]+t_1
\]
or 
\[
\Theta[P_{j-1}-P_{j-2}]=\Theta[P_{j-1}-P_{j-2}-1]+t_1-b
\]
holds. 
On the other hand, we have $\Theta[P_{j-1}-P_{j-2}]=M_{j-2}$ by the induction hypothesis for (ii), and therefore 
\[
\Theta[P_{j-1}-P_{j-2}]-t_1=t_{i-1}+(c_1-2)t_1+\dots+(c_{j-2}-2)t_{j-2} >0.
\] 
Since $\Theta[P_{j-1}-P_{j-2}-1] <b$, it follows that $\Theta[P_{j-1}-P_{j-2}]=\Theta[P_{j-1}-P_{j-2}-1]+t_1$. 
Therefore, we have 
\[
\Theta[P_j-1]=t_{i-1}+(c_1-2)t_1+\dots+(c_{j-2}-2)t_{j-2}+(c_{j-1}-1)t_{j-1},
\]
and hence $\Theta[P_j]=t_{i-1}+t_j$ by Lemma \ref{generalxj}.  
Next, let $\kappa >1$. We first show that 
$\Theta[(\kappa-1)P_j+{\varepsilon}]=\Theta[{\varepsilon}]+(\kappa-1)t_j$ holds for any $1 \leq \varepsilon < P_j$. 
Since we have $t_{i-1}+(\kappa-1)t_j=\Theta[{(\kappa-1)P_j}]$ by the induction hypothesis for Step 1, it suffices to check that 
$\Theta[{\varepsilon}]+(\kappa-1)t_j < b$ holds concerning Remark \ref{e}.  Indeed, we have 
\begin{align*}
b-\{\Theta[{\varepsilon}]+(\kappa-1)t_j\} &\geq  b-\{M_{j-1} + (c_j-2)t_j\}\\
&=(c_{j+1}-2)t_{j+1}+ \dots + (c_{l}-2)t_{l}+(t_{l}-t_{l+1})-t_{i-1}\\
& \geq t_{l}+(t_{l}-t_{l+1})-t_{i-1}
 > 0. 
\end{align*}
Taking $\varepsilon=P_j-1$, we obtain  
$\Theta[{\kappa P_j-1}]=\Theta[{P_j-1}]+(\kappa-1)t_j$. 
Therefore, we have $\Theta[\kappa P_j-1]+t_1-b=t_{i-1}+\kappa t_j >0$,  and hence $\Theta[\kappa P_j]=t_{i-1}+\kappa t_j$.  

\emph{Step 2.} In this step, we prove that 
\[
\Theta[{(c_j-1)P_j+\varepsilon}]=\Theta[{\varepsilon}]+(c_j-1)t_j
\]  
holds for any $0 < \varepsilon < P_j-P_{j-1}$, which completes the proof of (i). 
If $c_1=\dots =c_{j-1}=2$, then we have $P_j-P_{j-1}=1$, and there is nothing to prove. Suppose otherwise. 
Then, as in Step1, it is enough to show that 
\[
\max\{\Theta[{\varepsilon}]\; :\; 0 < \varepsilon < P_j-P_{j-1}\}+(c_j-1)t_j < b
\]
holds. 
Let 
$u=\max\{j'\; :\; 1 \leq j'\leq j-1,\; c_{j'} >2\}$.  
Then we have $ P_j-P_{j-1}=P_{u+1}-P_{u}=(c_u-2)P_u+(P_u-P_{u-1})$, and 
we see that 
\begin{align*}
&\max\left\{\Theta[\varepsilon] : 0 \leq \varepsilon < P_{u+1}-P_u\right\}\\
&\qquad \quad  =\max\left\{(c_u-3)t_u+M_{u-1},\; (c_u-2)t_u+\max\left\{\Theta[\varepsilon]: 0 \leq \varepsilon<P_u-P_{u-1}\right\}\right\}.
\end{align*}
Notice that 
\begin{align*}
& \max\left\{\Theta[\varepsilon]: 0 \leq \varepsilon<P_u-P_{u-1}\right\}\\
& \; =\max\left\{(c_{u-1}-3)t_{u-1}+M_{u-2},\; (c_{u-1}-2)t_{u-1}+\max\left\{\Theta[\varepsilon] : 0 \leq \varepsilon < P_{u-1}-P_{u-2}\right\}\right\},
\end{align*}
and that 
\[
(c_u-3)t_u+M_{u-1} -\{(c_u-2)t_u+(c_{u-1}-3)t_{u-1}+M_{u-2}\}=t_{u-1}-t_u>0.
\] 
These yield that 
\begin{align*}
& \max\left\{\Theta[\varepsilon] : 0 \leq \varepsilon < P_{u+1}-P_u\right\}\\
&\qquad \; \quad = \max\left\{(c_u-3)t_u+M_{u-1},\; (c_u-2)t_u+\dots+(c_2-2)t_2+(c_1-2)t_1+t_{i-1}\right\} \\
& \qquad \; \quad =(c_u-3)t_u+M_{u-1}. 
\end{align*}  
Therefore, 
\begin{align*}
& b-\left\{\max\{\Theta[\varepsilon]: 0 \leq \varepsilon <P_j-P_{j-1}\}+(c_j-1)t_j\right\} \\
& \quad =b-\{(c_u-3)t_u+M_{u-1}+(c_j-1)t_j\} \\
&\quad =t_u+(c_{u+1}-2)t_{u+1}+\dots+(c_j-2)t_j+\dots+(c_l-2)t_l+t_l-t_{l+1}-t_{i-1}-(c_j-1)t_j\\
& \quad \geq t_u+t_l+t_l-t_{l+1}-t_j-t_{i-1} >0. 
\end{align*}

This completes the proof of (i). 

\emph{Step 3.} In this last step, we give the proof of (ii). 
First, we show that $M_j=t_{i-1}+b-t_j+t_{j+1}$.  
Note that we have 
\[
M_j=\max\{M_{j-1},\; \max\{\Theta[x]\; :\; P_j \leq x < P_{j+1}\}\}.
\] 
Set 
\[
M_A=\max\{\Theta[x]\; :\; P_j \leq x< (c_j-1)P_j\},
\] 
and 
\[
M_B=\max\{\Theta[x]\; :\; (c_j-1)P_j \leq x< P_{j+1}\}.
\]  
Then we see that 
\[
M_A=(c_j-2)t_j+M_{j-1}=t_{i-1}+b-t_j+t_{j+1}
\] 
and that 
\[
M_B= (c_j-1)t_j+\max\{\Theta[\varepsilon] : 0 \leq \varepsilon < P_j-P_{j-1}\}.
\]  
Therefore it follows that $M_j=\max\{M_{j-1},\; \max\{M_A, \; M_B\}\}=\max\{M_A,\; M_B\}$. 
If $c_1=\dots=c_{j-1}=2$, then we have $M_B=(c_j-1)t_j+t_{i-1}$, and hence 
$M_A-M_B=b-t_{j-1}>0$. Thus, we get $M_j=M_A$. 
Suppose that $c_{j'}>2$ holds for some $1 \leq j' \leq j-1$, and take $u$ as in Step 2. Then, we have $M_B=(c_j-1)t_j+(c_u-3)t_u+M_{u-1}$. In a similar manner as above we see that $M_A-M_B \geq t_u-t_j>0$, and therefore $M_j=M_A$. 
Finally, we show $M_j=\Theta[P_{j+1}-P_j]$.  Since $t_{i-1}+(c_j-2)t_j=\Theta[(c_j-2)P_j]$ and $(c_j-2)t_j+\Theta[P_j-P_{j-1}]=(c_j-2)t_j+M_{j-1}=M_j <b$, it follows that
 \[
(c_j-2)t_j+\Theta[{P_j-P_{j-1}}]=\Theta[(c_j-2)P_j+P_j-P_{j-1}]=\Theta[{P_{j+1}-P_j}].
\] 
This completes the proof of the lemma. 
\end{proof}
We can show the next lemma by following a similar way as in Lemma \ref{keyprop}.  
\begin{lemma}\label{j=varpi}
Let $P_{l} \leq x <  P_{l+1}-P_{l}=P_i-P_{i-1}$, and 
denote by $\kappa$ (resp. $\varepsilon$) the quotient (resp. the remainder) of $x$ divided by $P_l$, i.e., $x=\kappa P_l+\varepsilon$. Then, we have $\Theta[x]=\Theta[\varepsilon]+\kappa t_l$.  In particular, we have $\Theta[x]> t_{i-1}$. 
\end{lemma}
As an immediate consequence of Lemmas \ref{keyprop} and  \ref{j=varpi}, we get the following. 
\begin{corollary}\label{elementary}
With the assumptions (C1) and (C2), we have $\Theta[x] \geq  t_{i-1}$ for any $0 \leq x < P_i-P_{i-1}$. 
Moreover, we have $\Theta[x]=t_{i-1}$ if and only if $x=0$. 
\end{corollary}
\begin{corollary}\label{cor1}
Let $1 \leq i \leq r+1$. Then we have $\Rem[tx, b] \leq b+t_i-t_{i-1}$ for any $0 < x< P_i$. 
\end{corollary}
\begin{proof}
We have $P_{j-1} \leq x < P_j$ for some $1 < j<i$. 
If $c_1=\dots =c_{j-1}=2$, then $P_{j'}=j'$, $Q_{j'}=j'-1$, and $t_{j'}=t{j'}-({j'}-1)b$ hold for any $1\leq j' \leq j$.  It follows that $x=j-1$ and $\Rem[t(j-1),b]=t_{j-1} \leq t=b+t_i-t_{i-1}$. 
Next, suppose that we have $c_l >2$ and $c_{l+1}=\dots =c_{j-1}=2$ for some $1 \leq l \leq j-1$. 
Then we have $\Rem[tx,b]=\Theta[x-P_{j-1}]$. 
By the proof of Lemma \ref{keyprop}, we see that the following holds: 
\[
\max\{\Theta[y]\; :\; 0 \leq y < P_j-P_{j-1}=P_{l+1}-P_l\}=
(c_l-3)t_l+M_{l-1}.
\]
Therefore, 
\begin{align*}
b+t_i-t_{i-1}-\Rem[tx,b] & \geq b+t_i-t_{i-1}-\{(c_l-3)t_l+M_{l-1}\}\\
& \geq (c_{i-1}-2)t_{i-1}-(c_l-3)t_l >0.
\end{align*}
\end{proof}
\begin{corollary}\label{cor of elementary}
Let $1 \leq i \leq r+1$. Then, we have $\Rem[tP_i, b]=t_i$. 
Moreover, we have $\Rem[tx, b] \geq t_{i-1}$ for any $0 < x< P_i$. 
\end{corollary}
\begin{proof}
We have seen that $\Rem[tP_i, b]=t_i$ holds for any $i$. 
Let $i >1$. 
As in the proof of Corollary \ref{cor1}, we have $P_{j-1} \leq x <P_j$ for some $1 < j < i$. 
If $c_1= \dots =c_{j-1}=2$, then we have $P_j=j$, and hence $\Rem[tx,b]=\Rem[tP_{j-1}, b]=t_{j-1} \geq t_{i-1}$.  
Otherwise, we have 
$\Rem[tx, b]=\Theta[x-P_{j-1}] \geq t_{j-1}$ by Corollary \ref{elementary}.  
\end{proof}
\begin{proof}[Proof of Proposition \ref{long}]
Let $\lambda=(n_i,c,\w) \in \Lambda_{(n_i,0)}$, and write $f_{\lambda}=X_0^{d_0}X_1^{d_1}X_3^{d_3}$. 
First, suppose that  $i=0$. 
By Example \ref{positive}, we have 
$R_{(n_0,0)}=R^0_{n_0} \oplus \left(\bigoplus_{c >0} R^c_{n_0}\right)$ and $R^0_{n_0}=\langle X_0^{n_0},\; X_1X_3 \rangle$. 
Therefore, it suffices to show that $f_{\lambda}$ is contained in the ideal $(X_0^{n_0},\; X_1X_3)$. 
Notice that $R^c_{n_0}$ decomposes as 
\[
R^c_{n_0}=R_{(n_0,c,\w_{(n_0,c)})} \oplus \left(\bigoplus_{\w > \w_{(n_0,c)}} R_{(n_0,c,\w)}\right).
\]  
If $\w > \w_{(n_0,c)}$, then we have $f_{\lambda} \in (X_0^{n_0})$ by Lemma \ref{min}, since $d_0=n_0+\w$. 
Suppose that $\w=\w_{(n_0,c)}$.  
Then we have $q-p+pd_1-qd_3=d_0=n_0+\w < q-p$, and thus $d_1 >0$ and $d_3>0$. 
Therefore,  $f_{\lambda} \in (X_1X_3)$. 

Next, suppose that $1 \leq i \leq r+1$. By Example \ref{positive}, we have 
\[
R_{(n_i, 0)}=R^0_{n_i} \oplus \left(\bigoplus_{
\substack{c=mx \\ 0 <x< P_i}} R^c_{n_i}\right) 
\oplus R^{mP_i}_{n_i} \oplus \left(\bigoplus_{
\substack{c=mx \\
P_i < x}} R^c_{n_i} \right), \quad R^0_{n_i}=R_{(n_i,0,\w_{(n_i,0)})}=\langle X_0^{n_i} \rangle.
\] 
We may assume that neither $c=0$ nor $c=mP_i$. 
If $\w > \w_{(n_i, c)}$, then we have  $f_{\lambda} \in (X_0^{n_i})$ as above. 
Thus, concerning that  $R_{(n_i,mP_i,\w_{(n_i,mP_i)})}=\langle X_1^{e_i}X_3^{l_i} \rangle$, we are left to show that if $\w=\w_{(n_i,c)}$ then $f_{\lambda}$ is contained in the ideal $(X_0^{n_i}, X_1^{e_i}X_3^{l_i})$. 
We first consider the case when $0 < x < P_i$ and show that $\w_{(n_i, c)} \geq 0$, which implies that $f_{\lambda} \in (X_0^{n_i})$. 
By Corollary \ref{preremainder}, we have 
$\w_{(n_i, c)} \geq 0$ if and only if $\Rem[pc, q-p]+n_i < q-p$. 
Note that we have 
\begin{align}
\Rem[pc, q-p]+n_i < q-p 
& \Leftrightarrow \Rem[pc+n_i, q-p] \geq n_i  \notag \\
&\Leftrightarrow \Rem\left[\frac{pc+n_i}{k}, b\right] \geq t_i. \label{equiv1}
\end{align}
We also have 
\begin{equation}
\Rem\left[\frac{pc+n_i}{k}, b \right]=\Rem\left[t(P_i-x), b\right], \label{equiv2}
\end{equation}
since we see by using equations \eqref{mp} and \eqref{t} that 
\begin{equation*}
\frac{pc+n_i}{k}=x \left\{(\alpha+1)b-t\right\}+(tP_i-bQ_i)
 \equiv t(P_i-x)\quad (mod\; b).
\end{equation*}
Therefore it follows from Corollary \ref{cor of elementary} that $\w_{(n_i,c)} \geq 0$. 
Next, we consider the case when $x > P_i$ and 
show that $f_{\lambda} \in (X_1^{e_i}X_{3}^{l_i})$. 
Set $\w'=-n_i+q(c-mP_i)$, and $\w''=-n_i+p(c-mP_i)$. 
First, suppose that $d_1 < e_i$. 
Then we have 
$qc-\w_{(n_i, c)} < (q-p)e_i=n_i+qmP_i$, 
and hence 
$\w_{(n_i, c)} > \w'$. 
It follows that $0 \leq pc-\w_{(n_i,c)} < pc-\w'=n_i+qmP_i-c(q-p)$. Therefore, all of the following are positive integers: 
\[
n_i+\w'=q(c-mP_i), \quad
\frac{qc-\w'}{q-p}=\frac{n_i+qmP_i}{q-p}, \quad 
\frac{pc-\w'}{q-p}=\frac{n_i+qmP_i}{q-p}-c. 
\]
This implies that $(n_i, c, \w')\in \tilde{\mu}^{-1}(n_i, c)$, which contradicts to the minimality of $\w_{(n_i, c)}$. 
Next, suppose that $d_{3} < l_i$.  
Then we have 
$pc -\w_{(n_i, c)} < (q-p)l_i=n_i+pmP_i$, and hence $\w_{(n_i, c)} > \w''$.  
In a similar manner, we see that this implies $(n_i, c, \w'') \in \tilde{\mu}^{-1}(n_i, c)$, which is a contradiction. Therefore, $d_1 \geq e_i$ and $d_3 \geq l_i$. 
\end{proof}
\begin{corollary}\label{cor2}
We have $\Rem[pmx+n_i, q-p]=n_i+\Rem[pmx, q-p]$ for any $ 0 < x< P_i$.  
\end{corollary}
\begin{proof}
By the proof of Proposition \ref{long}, we see that $\Rem[pmx+n_i, q-p] \geq n_{i-1}$. 
On the other hand, we have 
\[
\Rem[pmx+n_i, q-p]=
\begin{cases}
n_i+\Rem[pmx, q-p] \quad (\mbox{if}\; n_i+\Rem[pmx, q-p] < q-p)\\
n_i+\Rem[pmx, q-p]-q+p \quad (\mbox{otherwise}). 
\end{cases}
\]
Therefore we deduce that $\Rem[pmx+n_i, q-p]=n_i+\Rem[pmx, q-p]$, since otherwise we have 
$n_{i-1} \leq n_i+\Rem[pmx, q-p]-q+p < n_i < n_{i-1}$.  
\end{proof}
\section{Second step towards the proof of Theorem \ref{main theorem}}\label{s-second}
In this section, we construct an equivariant morphism 
\[
\xymatrix{
\Psi : \mathcal H \ar[r] & \E \times \P(V^{\vee})
}
\]
that satisfies $\Psi(\mathcal H^{main})=\Phi(\widetilde{\E'}) \cong \widetilde{\E'}$ (Proposition \ref{image}). 
First, we see by Theorem \ref{generator2} that we can construct an equivariant morphism 
\[
\xymatrix{
\eta_{n_i,0} : \mathcal H \ar[r] & \Gr(1, F_{n_i,0}^{\vee}) \cong \P(F_{n_i,0}^{\vee})
}
\]
for each $0 \leq i \leq r$. 
Set  
\[
\xymatrix{
\Delta:=\gamma \times \prod_{0 \leq i \leq r} \eta_{n_i, 0} : \mathcal H \ar[r] &  E_{l,m} \times \prod_{0 \leq i \leq r} \P(F_{n_i,0}^{\vee}),
}
\]
and let
\[
\xymatrix{
\iota : \prod_{0 \leq i \leq r} \P(F_{n_i,0}^{\vee}) \ar@{^{(}-_>}[r] & \P(V'^{\vee})
}
\]
be the Segre embedding, where 
$V':=F_{n_0,0} \otimes F_{n_1,0} \otimes \dots \otimes F_{n_r,0}$. 
We see that $V'$ coincides with 
\[
\bigoplus A(e_0) \otimes B(l_0) \otimes A(e_{i_1}) \otimes B(l_{i_1}) \otimes \dots A(e_{i_s}) \otimes B(l_{i_s}) \otimes C(n_{j_1}) \otimes \dots \otimes C(n_{j_u}),
 \]
where the sum runs over $\{i_1,\; \dots ,\; i_s,\; j_1,\; \dots ,\; j_u\}
=\{1,\; \dots ,\; r\}$ such that $i_1 < \dots < i_s$ and $j_1 < \dots < j_u$. 
\begin{remark}\label{gordan}
As in Remark \ref{explicit}, we denote by $V(n)$ the irreducible $SL(2)$-representation of 
highest weight $n$. 
For any partition $n=\mu_1+ \dots + \mu_s$, 
 the tensor representation $V(\mu_1) \otimes \dots \otimes V(\mu_s)$ contains an irreducible representation $V(\mu_1,\; \dots,\; \mu_s)$ isomorphic to $V(n)$ by the Clebsch--Gordan theorem. 
For each $0 \leq i \leq n$, set 
\[
\phi_i:=\frac{1}{\begin{pmatrix}
n \\
i
\end{pmatrix}}
\sum_{\substack{
i_1 + \dots + i_s=i \\
0 \leq i_1 \leq \mu_1\\
\dots \\
0 \leq i_s \leq \mu_s
}} 
\begin{pmatrix}
\mu_1\\
i_1
\end{pmatrix}
 \dots 
\begin{pmatrix}
\mu_s \\
i_s
\end{pmatrix}
X^{\mu_1-i_1}Y^{i_1} \otimes \dots \otimes X^{\mu_s-i_s}Y^{i_s}
\in V(\mu_1) \otimes \dots \otimes V(\mu_s).
\]
Then, $\{\phi_0,\; \dots ,\; \phi_n\}$ forms a basis of $V(\mu_1,\; \dots,\; \mu_s)$. 
On the other hand, we can take $\{X^{n-i}Y^i\; :\; 0 \leq i \leq n\}$ as a basis of $V(n)$, and the linear map
\[
V(n) \to V(\mu_1,\; \dots,\; \mu_s),\quad X^{n-i}Y^i \mapsto \phi_i
\]
is an $SL(2)$-equivariant isomorphism.
\end{remark}
Let us consider the submodule 
\[
\widetilde{V}:=\bigoplus_{1 \leq i \leq r}
A(e_0,\; e_{1},\; \dots,\; e_{i}) \otimes B(l_0,\; l_{1},\;  \dots,\; l_{i}) \otimes C(n_{i+1}) \otimes \dots \otimes C(n_{r})
\]
of $V'$, where $A(e_0, e_1, \dots, e_i) \cong V(e_0, e_1,  \dots, e_i)$ (resp. $B(l_0, l_{1}, \dots, l_{i}) \cong V(l_0, l_{1}, \dots, l_{i})$) stands for the irreducible representation of highest weight $e_0+e_1+ \dots+ e_i$ (resp. $l_0+l_{1}+ \dots+ l_{i}$ ) contained in $A(e_0) \otimes A(e_1) \otimes \dots \otimes A(e_i)$ (resp. $B(l_0) \otimes B(l_{1}) \otimes \dots \otimes B(l_{i})$) in the sence of Remark \ref{gordan}. 
Since $V \subset \Gamma(\widetilde{\E'}, \mathcal O(\delta))$ coincides with
\[
\bigoplus_{1 \leq i \leq r} A(e_0+e_{1}+ \dots+e_{i}) \otimes B(l_0+l_{1}+ \dots + l_{i}) \otimes C(-(n_0+n_{1} + \dots +n_{i})), 
\] 
we see that $V \cong \widetilde{V}$, where 
 the isomorphism 
\[
C(-(n_0+n_{1}+ \dots +n_{i})) \cong C(n_{i+1} + \dots + n_{r}) \cong C(n_{i+1}) \otimes \dots \otimes C(n_{r})
\] 
is given by multiplying $X_0^{n_0+n_1 + \dots + n_r}$. 
\begin{example}\label{formerex}
Let $l=p/q=1/4$, and $m=2$ as in Example \ref{fex}. 
Then, we have $k=1$, $a=2$, $b=3$, $\alpha=0$, $\beta=2$, and $t=1$. 
Therefore, the Hirzebruch--Jung continued fraction of $b/t$ is $b/t=c_1=3$, and we have 
$P_0=0$, $Q_0=-1$, $P_1=1$, $Q_1=0$, 
$P_2=c_1=3$, and $Q_2=1$. 
Thus, we get $\rho_0=\mbox{\bf u}_2$, $\rho_1=-\mbox{\bf u}_1+\mbox{\bf u}_2$, and $\rho_2=-3\mbox{\bf u}_1+2\mbox{\bf u}_2$, and the  
maximal cones of the colored fan of $\widetilde{E_{\frac{1}{4},2}'}$ are the following: 
\[
\mathcal C_1=\mathbb Q_{\geq 0} \rho_0 + \mathbb Q_{\geq 0} \rho_1, \quad \mathcal C_2=\mathbb Q_{\geq 0} \rho_1 + \mathbb Q_{\geq 0} \rho_2.
\] 
Also, we have $(e_0, l_0, n_0)=(1,1,3)$, $(e_1, l_1, n_1)=(3,1,1)$, and $(e_2, l_2, n_2)=(8,2,0)$. Thus we get $f_0=ZW$, $f_1=Z^4W^2$, and $f_2=Z^{12}W^4$ by definition, and therefore 
\begin{align*}
V&=\langle (SL(2) \times \C^*) \cdot ZW\rangle \oplus \langle (SL(2) \times \C^*) \cdot Z^4W^2\rangle\\
&\cong \langle X, Z\rangle  \otimes \langle Y,W\rangle \oplus \langle X^4, X^3Z, X^2Z^2, XZ^3, Z^4\rangle  \otimes \langle Y^2, YW, W^2\rangle \\
& \cong V(1) \otimes V(1) \oplus V(4) \otimes V(2). 
\end{align*}
We have $V'=F_{n_0, 0} \otimes F_{n_1, 0}$, where 
\[
F_{n_0,0}=A(1) \otimes B(1)=\langle X_1, X_2\rangle  \otimes \langle X_3, X_4\rangle
\]
and 
\[
F_{n_1, 0}=A(3) \otimes B(1) \oplus C(1)=\langle X_1^3, X_1^2X_2, X_1X_2^2, X_2^3\rangle  \otimes \langle X_3, X_4\rangle \oplus \langle X_0\rangle.
\]
Furthermore, we have 
$\widetilde{V}=A(1,3) \otimes B(1,1) \oplus 
A(1) \otimes B(1)$, 
where $A(1,3)$ is a subrepresentation of $A(1) \otimes A(3)$ spanned by the following vectors:
\begin{align*}
&X_1 \otimes X_1^3, \quad \frac{1}{4}(X_2 \otimes X_1^3+3X_1 \otimes X_1^2X_2) ,\quad \frac{1}{2}(X_2 \otimes X_1^2X_2+X_1 \otimes X_1X_2^2),\\
& \frac{1}{4}(3X_2 \otimes X_1X_2^2+X_1\otimes X_2^3),\quad 
X_2 \otimes X_2^3.
\end{align*}
Also, $B(1,1)$ is a subrepresentation of $B(1) \otimes B(1)$ spanned by the following vectors:
\[
X_3 \otimes X_3, \quad \frac{1}{2}(X_3 \otimes X_4+X_4 \otimes X_3),\quad X_4 \otimes X_4.
\]
\end{example}
Now, set  
\[
\xymatrix{
{\Psi}' :=(\id_{\E} \times \iota) \circ \Delta:  \mathcal H \ar[r] & \E \times \P(V'^{\vee}),
}
\]
and consider the projection
\[
\xymatrix{
\pr : \E \times \P(V'^{\vee}) \ar@{.>}[r] &\E \times \P(\widetilde{V}^{\vee}).
}
\]
\begin{proposition}\label{morphism}
The restriction
$\pr |_{{\Psi}'(\mathcal H)}$ of the rational map $\pr$ to the image of ${\Psi}'$ is a morphism. 
\end{proposition}
\begin{proof}
Let
\[
[({X_2X_4})^{\vee} : ({X_1X_4})^{\vee}: (X_2X_3)^{\vee}: (X_1X_3)^{\vee}]
\]
and  
\[
[({X_0^{n_i}})^{\vee} : ({X_2^{e_i}X_4^{l_i}})^{\vee} : \dots : (X_2^{e_i-e}X_1^eX_4^{l_i-l}X_3^l)^{\vee}: \dots : ({X_1^{e_i}X_3^{l_i}})^{\vee}] \quad (1 \leq i \leq r)
\]
be the coordinate of $\P(F_{n_0,0}^{\vee})$ and $\P(F_{n_i,0}^{\vee})$, respectively. 
Suppose that there is a point $[I] \in \mathcal H$
 such that $\pr$ is not defined at $\Psi'([I])$. 
Let 
\[
\eta_{n_0,0}([I])=[t^{(0)}_{0,0}: t^{(0)}_{e_0,0}: t^{(0)}_{0,l_0}: t^{(0)}_{e_0,l_0}]
\]
and 
\[
\eta_{n_i,0}([I])=[u^{(i)}: t^{(i)}_{0,0}: \dots : t^{(i)}_{e,l}: \dots: t^{(i)}_{e_i, l_i}] \quad  (1 \leq i \leq r).
\]
By \eqref{s1s2}, we have $s_1X_1+s_2X_2 \in I$ for some $(s_1,s_2) \neq (0,0)$.  
Since ${\Psi}'$ is $SL(2)$-equivariant, we may assume that $X_2 \in I$. 
The subrepresentation 
$A(e_0, e_1, \dots, e_r) \otimes B(l_0, l_1, \dots, l_r) \subset \widetilde{V}$ contains 
$X_1^{e_0} \otimes X_1^{e_1} \otimes \dots \otimes X_1^{e_r} \otimes X_3^{l_0} \otimes X_3^{l_1} \otimes \dots \otimes X_3^{l_r}$, 
and therefore we have  
$t^{(0)}_{e_0,l_0} t^{(1)}_{e_1, l_1} \cdots t^{(r)}_{e_r, l_r}=0$ by the assumption on the ideal $I$. 
Let $j=\min\{i \; :\; t^{(i)}_{e_i,l_i}=0, \; 0 \leq i \leq r\}$. 
Then we have 
$X_1^{e_j}X_3^{l_j} \in I$ 
by the construction of $\eta_{n_j,0}$, and hence 
$X_1^{e_i}X_3^{l_i} \in I$  for every $i \geq j$. 
Next we have 
$s_3X_3+s_4X_4 \in I$ for some $(s_3, s_4) \neq (0,0)$ by \eqref{s3s4}. 
Namely, one of the following holds: 
(a) $s_3 \neq 0,\; s_4 \neq 0$; 
(b) $s_3=0,\;  s_4 \neq 0$; 
(c) $s_3 \neq 0,\; s_4=0$. 
Suppose that we are in the case (a). Then, by multiplying $X_1^{e_j}X_3^{l_j-1}$ to $s_3X_3+s_4X_4$, we get $X_1^{e_j}X_3^{l_j-1}X_4 \in I$. By continuing in this way, we finally obtain 
\begin{equation*}
X_1^{e_j-e}X_2^eX_3^{l_j-l}X_4^l \in I  \quad (0 \leq \forall e \leq e_j,\quad 0 \leq  \forall l \leq l_j)\label{cont1}
\end{equation*}
concerning $X_2 \in I$. 
Lastly, we pay attention to the vector
\[
X_1^{e_0} \otimes X_1^{e_1} \otimes \dots \otimes X_1^{e_{j-1}} \otimes X_3^{l_0} \otimes X_3^{l_1} \otimes \dots \otimes X_3^{l_{j-1}} \otimes X_0^{n_j} \otimes \dots \otimes X_0^{n_r}
\]
contained in the following subrepresentation of $\widetilde{V}$: 
\[ 
A(e_0, e_1, \dots, e_{j-1}) \otimes B(l_0,  l_1, \dots,  l_{j-1}) \otimes C(n_j) \otimes \dots \otimes C(n_r).
\]
Likewise, we have 
$t^{(0)}_{e_0,l_0} t^{(1)}_{e_1, l_1} \cdots t^{(j-1)}_{e_{j-1}, l_{j-1}} u^{(j)} \cdots u^{(r)}=0$ by the assumption on $I$. 
This implies that $u^{(j)} \cdots u^{(r)}=0$ by the minimality of $j$, and 
therefore we have $X_0^{n_j} \in I$. 
Thus, we get $F_{n_j,0} \subset I$. Then it follows from Theorem \ref{generator2} that $\dim (\C[H_{q-p}]/I)_{(n_j, 0)}=0$, which 
 contradicts to $[I] \in \mathcal H$. 
\end{proof}
Combining the above discussion, we obtain the following equivariant morphism: 
\[
\xymatrix{\Psi : 
\mathcal H \ar[r]^{{\Psi}'\quad \; \;} & \E \times \P(V'^{\vee}) \ar@{.>}[r]^{\pr} &  \E \times \P(\widetilde{V}^{\vee}) \ar[r]^{\sim\; \;} & \E \times \P(V^{\vee}).
}
\]
\begin{proposition}\label{image}
We have 
$\Psi(\mathcal H^{main}) =\Phi(\widetilde{\E'})$.
\end{proposition}
\begin{proof}
Let $y \in \widetilde{\E'}$ be the fiber of 
 $\pi(x) \in \mathfrak U \subset \E$ under the canonical birational morphism $\widetilde{\E'} \to \E$, where $x=(1,1,0,0,1) \in H_{q-p}$.  
Then, concerning Remark \ref{monomial}, we have $\Phi(y)=(\pi(x), v)$, where $v$ is a point in $\P(V^{\vee})$ whose coordinates are all $0$ except for the ones corresponding to the bases
\[
(X^{e_0+e_1+ \dots+ e_i}W^{l_0+l_1+ \dots +l_i})^{\vee} \quad (1 \leq i \leq r).
\]
On the other hand, it follows from the definition of $I_1$ and the construction of $\eta_{n_i,0}$ that  
\[
\eta_{n_0,0}([I_1])=\langle (X_1X_4)^{\vee}\rangle, \qquad \eta_{n_i,0}([I_1])=\langle {(X_0^{n_i})}^{\vee}+(X_1^{e_i}X_4^{l_i})^{\vee}\rangle \quad (1 \leq i \leq r).
\]
Therefore, we get $\Psi([I_1])=\Phi(y)$, and hence the proposition. 
\end{proof}
Summarizing, we get the following equivariant commutative diagram:
\[
\xymatrix@C=36pt@R=-4pt{
\mathcal H \ar[rrrr]^{\; \; \; \Psi\qquad} & &&  &\E \times \P(V^{\vee}) \\
\rotatebox{90}{$\subset$} &&  & &\rotatebox{90}{$\subset$} \\
\mathcal H^{main} \ar[rrdddddddddddddddddddd]_{\gamma|_{\mathcal H^{main}}} \ar[rrddddd]_{\; \; \; \;  \; \; \; \; \psi |_{\mathcal H^{main}}} \ar[rrrr]^{\; \; \; \;\;\Psi|_{\mathcal H^{main}}} &  & & & \Phi(\widetilde{\E'}) \cong \widetilde{\E'}\ar[llddddd]\ar[lldddddddddddddddddddd] \\
& &  && \\
& & & &\\
& & & &\\
& & & &\\
& & \E' \ar[ddddddddddddddd]& &\\
& & & &\\
& & & &\\
& && &\\
&&  &&\\
& && &\\
& && &\\
& && &\\
& &&&\\
& &&&\\
&&&&\\
& &&&\\
&&&&\\
& &&&\\
&&&&\\
& & \E & &
}
\]

\section{Calculation of ideals}\label{s-idealo}
For each $1 \leq i \leq r$, we consider the ideals 
\[
J_1^i:=(X_0^{n_{i-1}},\; X_2,\;  X_4,\; X_0^{n_i}-X_1^{e_i}X_3^{l_i})+K
\]
and 
\[
J_0^i:=(X_0^{n_{i-1}},\; X_2,\; X_4,\; X_1^{e_i}X_3^{l_i})+K
\]
of $A=\C[X_0, X_1, X_2, X_3, X_4]$, 
where $K$ is the ideal generated by elements of the form:
\[
X_0^{pu_1-qu_2}X_1^{u_1}X_3^{u_2}, \quad (u_1, u_2) \in M^+_{l,m} \setminus \{(0,0)\}. 
\]
Also, we define 
\[
J_1^{r+1}:=(X_0^{n_r},\; X_2,\; X_4,\; X_0^{n_{r+1}}-X_1^{e_{r+1}}X_3^{l_{r+1}})=(X_0^k,\; X_2,\; X_4, \;1-X_1^{aq}X_3^{ap}),
\]
and 
\[
J_0^{r+1}:=(X_0^{n_r}, \; X_2,\; X_4,\; X_1^{e_{r+1}}X_3^{l_{r+1}})=(X_0^k,\; X_2,\; X_4,\; X_1^{aq}X_3^{ap}).
\]
We will see in \S \ref{final} that every ideal of a closed point in $\mathcal H^{main}$ can be described as an $SL(2)$-translate of 
$I_1$, $I_0$, $J^i_1$, $J^i_0$, $J^{r+1}_1$, or  $J^{r+1}_0$.
\begin{remark}\label{another}
Let us define $F_j=f_{(0,mj,\w_{(0,mj)})}$ for each $1 \leq j \leq b-1$. Then, 
$J^i_1$ and $J^i_0$ coincide with 
\[
(X_0^{n_{i-1}},\; X_2, \;X_4, \;X_0^{n_i}-X_1^{e_i}X_3^{l_i}, \;F_1,\; \dots,\; F_{b-1})
\]
and 
\[
(X_0^{n_{i-1}},\; X_2, \;X_4, \;X_1^{e_i}X_3^{l_i}, \;F_1, \;\dots, \;F_{b-1}),
\]
respectively.   
\end{remark}
\begin{example}\label{example1}
Let $l=p/q=1/4$, and $m=2$ as in Examples \ref{fex} and  \ref{formerex}.  
Then we have $F_1=X_0^2X_1^2$ and $F_2=X_0X_1^5X_3$, and 
the ideals in consideration are described as follows:
\begin{align*}
& I_1=(X_0^3-X_1X_4,\; X_2,  \;1-X_0^2X_1^2);\\
& I_0=(X_0^3-X_1X_4,\; X_2,  \;X_0^2X_1^2); \\ 
& J^1_1=(X_0^3,\; X_2,\; X_4, \;X_0-X_1^3X_3, \;X_0^2X_1^2, \;X_0X_1^5X_3);\\
& J^1_0=(X_0^3, \;X_2, \;X_4, \;X_1^3X_3, \;X_0^2X_1^2, \;X_0X_1^5X_3); \\
& J^2_1=(X_0, \;X_2, \;X_4, \;1-X_1^8X_3^2); \\ 
& J^2_0=(X_0,\; X_2, \;X_4,\; X_1^8X_3^2). 
\end{align*}
\end{example}
Set $\tilde{K}:=K/(X_2, X_4)$. 
For each $1 \leq i \leq r+1$, we define 
\[
\tilde{J}^i_0:=J^i_0/(X_2, X_4)=(X_0^{n_{i-1}}, \;X_1^{e_i}X_3^{l_i})+\tilde{K} \subset R
\]
and 
\[
\tilde{J}^i_1:=J^i_1/(X_2, X_4)=(X_0^{n_{i-1}}, \;X_0^{n_i}-X_1^{e_i}X_3^{l_i})+\tilde{K} \subset R.
\]  
\begin{theorem}\label{idealO2}
Let $1 \leq i \leq r+1$. Then,  $\dim (A/J^i_0)_{(n,d)}=\dim (R/\tilde{J}^i_0)_{(n,d)} \leq h(n,d)$ holds for any weight 
$(n,d) \in \mathbb Z \times \mathbb Z/m \mathbb Z$. 
\end{theorem}
\begin{theorem}\label{idealO}
Let $1 \leq i \leq r+1$. Then,  $\dim (A/J^i_1)_{(n,d)}=\dim (R/\tilde{J}^i_1)_{(n,d)} \leq h(n,d)$ holds for any 
weight $(n,d) \in \mathbb Z \times \mathbb Z/m \mathbb Z$. 
\end{theorem}
The proof of Theorems \ref{idealO2} and \ref{idealO} will be given after preparing a few lemmas. 
\begin{lemma}\label{invring}
Let  $\lambda=(n,c,\w) \in \Lambda_{(n,0)}$. 
If $n\geq 0$, $\w \geq 0$ and $c > 0$, then $f_{\lambda} \in \tilde{K}$.
\end{lemma}
\begin{proof}
Write $f_{\lambda}=X_0^{d_0}X_1^{d_1}X_3^{d_3}$.  
Then we have $f_{\lambda}=X_0^n(X_0^{d_0-n}X_1^{d_1}X_3^{d_3})$ concerning $d_0=n+\w$. 
The conditions $f_{\lambda} \in R_{(n,0)}$ and $X_0^n \in R_{(n,0)}$ imply that $X_0^{d_0-n}X_1^{d_1}X_3^{d_3} \in R^{\G}$. 
Since $0 <c=d_1-d_3$, it follows that $X_0^{d_0-n}X_1^{d_1}X_3^{d_3} \in \tilde{K}$. 
\end{proof}
\begin{lemma}\label{remainder}
Let $(n,c) \in \Lambda'$. 
Assume that $0 \leq n < q-p$, and that $c \geq 0$. Then the following properties are true. 
\begin{itemize}
\item [\rm(i)] We have $\w_{(0, c)}+n < q-p$ if and only if $\w_{(0,c)}=\w_{(n,c)}$. 
\item [\rm(ii)] We have $\w_{(0,c)}+n \geq q-p$ if and only if  $\w_{(0,c)}=\w_{(n,c)}+q-p$. 
\item [\rm(iii)] We have $\w_{(0,c)} \geq q-p-\beta$ if and only if  $\w_{(0,c+m)}=\w_{(0,c)}-q+p+\beta$.  
\item [\rm(iv)] We have $\w_{(0,c)} < q-p-\beta$ if and only if $\w_{(0,c+m)}=\w_{(0,c)}+\beta$.
\end{itemize}
\end{lemma}
\begin{proof}
First of all, concerning Lemmas \ref{1dim} and \ref{min}, we see that $0 \leq \w_{(0,c)} < q-p$ holds. Since we have 
$(n,c,\w_{(0,c)})=\mu\left(n+\w_{(0,c)}, \frac{qc-\w_{(0,c)}}{q-p}, \frac{pc-\w_{(0,c)}}{q-p}\right) \in \Lambda$, it follows that $(n,c,\w_{(0,c)}) \in \tilde{\mu}^{-1}(n,c)$. 
The if part is easy to check, so we prove the only if part. 

(i) follows from Lemma \ref{min}. 

(ii) If $n + \w_{(0,c)} \geq q-p$, then we have $\w_{(0,c)}-\w_{(n,c)}=x(q-p)$ for some $x \geq 1$.    
If $x >1$, then we have  
$\w_{(0,c)} > q-p$, which is a contradiction. 

(iii) Set $\w=\w_{(0,c)}-q+p+\beta$. 
Then we get $0 \leq \w < q-p$. 
Since we have 
$\frac{q(c+m)-\w}{q-p}=\frac{qc-\w_{(0,c)}}{q-p}+\alpha+m+1 >0$, 
$\frac{p(c+m)-\w}{q-p}=\frac{pc-\w_{(0,c)}}{q-p}+\alpha+1 >0$, and 
\[
(0,c+m,\w) =\mu\left(\w, \frac{q(c+m)-\w}{q-p}, \frac{p(c+m)-\w}{q-p}\right),
\]
it follows that $(0,c+m,\w) \in \tilde{\mu}^{-1}(0,c+m)$.  
Therefore, we have $\w=\w_{(0,c+m)}$. 

(iv) Set $\w'=\w_{(0,c)}+\beta$. In a similar way we see that $(0,c+m, \w') \in \tilde{\mu}^{-1}(0,c+m)$, and therefore we have $\w'=\w_{(0,c+m)}$. 
\end{proof}
The next lemma follows from Lemmas \ref{ni} and \ref{min}. 
\begin{lemma}\label{q-p}
Let $\lambda=(n,c,\w) \in \Lambda$. 
If $\w > \w_{(n,c)}$, then we have $f_{\lambda} \in (X_0^{n_{i-1}})$ for any $1 \leq i \leq r+1$. 
\end{lemma}
\begin{lemma}\label{bigger}
Let $(0,c) \in \Lambda'$ with $c=mx$, and suppose that we have $0 < x < P_i$ for some $1 \leq i \leq r+1$. Then, 
$n_{i-1}-n_i \leq  \w_{(0,c)} \leq q-p-n_{i-1}$. 
\end{lemma}
\begin{proof}
Concerning the proof of Corollary \ref{preremainder}, we have $\w_{(0,c)}=\Rem[pc, q-p]$, which coincides with $\Rem[pc+n_i, q-p]-n_i$ by Corollary \ref{cor2}.  
On the other hand, we have 
$n_{i-1} \leq \Rem[pc+n_i, q-p]  =k \Rem[t(P_i-x),b] \leq q-p+n_i-n_{i-1}$ 
by \eqref{equiv2} and Corollaries \ref{cor1}, \ref{cor of elementary}, and \ref{cor2}, and hence the lemma. 
\end{proof}
\begin{definition}
For each $c \in m \mathbb Z_{>0}$, we define $\lambda_c:=(q-p-\w_{(0,c)},\; c,\; \w_{(0,c)}-q+p)$. 
\end{definition}
\begin{remark}
By a direct calculation, we see that 
\[
f_{\lambda_c}=X_1^{\frac{qc-\w_{(0,c)}}{q-p}+1}X_3^{\frac{pc-\w_{(0,c)}}{q-p}+1}.
\] 
Also, by applying Lemma \ref{remainder} (ii) with $n=q-p-\w_{(0,c)}$, we have 
$\w_{(0,c)}-q+p=\w_{(q-p-\w_{(0,c)}, c)}$. 
\end{remark}
\begin{example}\label{ex of rem}
By Example \ref{zi} and Lemma \ref{remainder}, we have $\w_{(0,mP_i)}=q-p-n_i$, and therefore $\lambda_{mP_i}=(n_i, mP_i, \w_{(n_i, mP_i)})$ and $f_{\lambda_{mP_i}}=X_1^{e_i}X_3^{l_i}$.  
\end{example}
\begin{lemma}\label{cor of rem}
With the above notation, we have 
$f_{\lambda_{c'}} \in (f_{\lambda_{c}})$ if $c' \geq c$. 
\end{lemma}
\begin{proof}
Since $c, c' \in m \Z_{>0}$, we may assume that $c'=c+m$ . 
Then by \eqref{mp} and Lemma \ref{remainder} we have 
\[
f_{\lambda_{c'}}=
\begin{cases}
X_1^{\alpha+m+1}X_3^{\alpha+1}f_{\lambda_{c}}\qquad (\mbox{if}\; \w_{(0,c)} \geq q-p-\beta)\\
X_1^{\alpha+m}X_3^{\alpha}f_{\lambda_{c}} \qquad (\mbox{otherwise}). 
\end{cases}
\]
\end{proof}
\begin{corollary}\label{cor of rem1}
Let $\lambda=(n,c, \w_{(n,c)}) \in \Lambda_{(n,0)}$, and assume that $0 < c$ and  that $0 \leq n < q-p$. 
Then we have the following. 
\begin{itemize}
\item [\rm(i)] If $\w_{(0,c)}+n <q-p$, then $f_{\lambda} \in \tilde{K}$.
\item [\rm(ii)] If $\w_{(0,c)}+n \geq q-p$, then $f_{\lambda}=X_0^{n+\w_{(n,c)}}f_{\lambda_c}=X_0^{n+\w_{(0,c)}-q+p}f_{\lambda_c} $. 
\end{itemize}
\end{corollary}
\begin{proof}
Item (i) follows from Corollary \ref{remainder1} (i), Lemmas \ref{remainder} (i), and \ref{invring}. 
Item (ii) is a consequence of Lemma \ref{remainder} (ii) and the definition of $\lambda_c$. 
\end{proof}
\begin{lemma}\label{contain}
Let $\lambda=(n,c,\w_{(n,c)}) \in \Lambda_{(n,0)}$ with $c=mx$. Then, the following properties are true for any   $1 \leq i \leq r+1$.
\begin{itemize}
\item [\rm(i)] If $0 < x< P_i$ and $0 \leq n < n_{i-1}$, then $f_{\lambda} \in \tilde{K}$.
\item [\rm(ii)] If $x=P_i$ and $0 \leq n < n_i$, then $f_{\lambda} \in \tilde{K}$.
\item [\rm(iii)] If $x=P_i$ and $n_i \leq n < q-p$, then $f_{\lambda} \in (X_1^{e_i}X_3^{l_i})$. 
\item [\rm(iv)] If $x > P_i$ and $0 \leq n < q-p$, then $f_{\lambda} \in (X_1^{e_i}X_3^{l_i})+\tilde{K}$.
\item [\rm(v)] If $x > P_i$ and $0 \leq n <n_{i}$, then $f_{\lambda} \in \tilde{J}^i_1$. 
\item [\rm(vi)] Let $x >P_i$ and $n_i \leq n < n_{i-1}$. 
\begin{itemize}
\item [\rm(vi-1)] If $x$ is not a multiple of $P_i$, then $f_{\lambda} \in \tilde{J}^i_1$. 
\item [\rm(vi-2)] If $x$ is a multiple of $P_i$, then  $f_{\lambda}-f_{(n,c-mP_i, \w_{(n, c-mP_i)})} \in \tilde{J}^i_1$. 
\end{itemize}
\end{itemize}
\end{lemma}
\begin{proof}
(i) follows from Corollary \ref{preremainder}, Lemmas \ref{invring}, and \ref{bigger}. 

(ii) We have $\w_{(0,mP_i)}=q-p-n_i$ by Example \ref{ex of rem}, and therefore 
$f_{\lambda} \in \tilde{K}$ by Corollary  \ref{cor of rem1} (i). 

(iii) follows from applying Corollary \ref{cor of rem1} (ii) with $c=mP_i$. 

(iv) follows from Lemma \ref{cor of rem} and Corollary \ref{cor of rem1}. 

(v) Concerning Corollary \ref{cor of rem1}, we may assume that $n+\w_{(0,c)} \geq  q-p$. 
Set $n'=n+\w_{(0,c)}-q+p$. Then we have $f_{\lambda}=X_0^{n'}f_{\lambda_c}$. 
Also, by Lemma \ref{cor of rem}, $f_{\lambda_c}$ can be written as $f_{\lambda_c}= f_{\lambda_{mP_i}} f=X_1^{e_i}X_3^{l_i} f$ with some $f \in R^{c-mP_i}_{n-n'-n_i}$. 
Therefore, $
f_{\lambda}=X_0^{n'+n_i}f-X_0^{n'}f(X_0^{n_i}-X_1^{e_i}X_3^{l_i})$. 
Now, since $X_0^{n'+n_i}f \in R^{c-mP_i}_n$, we have $X_0^{n'+n_i}f=f_{\lambda'}$ with some $\lambda'=(n,c-mP_i, \w') \in \tilde{\mu}^{-1}(n,c-mP_i)$. 
If $\w'>\w_{(n,c-mP_i)}$, then we have $f_{\lambda'} \in (X_0^{n_{i-1}})$, and hence $f_{\lambda} \in \tilde
{J}^i_1$. 
Suppose that $\w'=\w_{(n,c-mP_i)}$.  
If $0 < c-mP_i \leq mP_i$, then we have $f_{\lambda'} \in \tilde{K}$ by (i) and (ii), and hence $f_{\lambda} \in \tilde{J}^i_1$. If $c-mP_i > mP_i$, we can apply the same process to $f_{\lambda'}$, and 
continuing in this way we finally obtain $f_{\lambda} \in \tilde{J}^i_1$. 

(vi) is an immediate consequence of the proof of (v). 
\end{proof}
\begin{lemma}\label{cor of bigger}
Let $\lambda=(n,c,\w_{(n,c)}) \in \Lambda_{(n,0)}$ with $c=mx$. 
Suppose that $P_j < x < P_i$, $n_j \leq n < n_{j-1}$,  and $n-n_j < n_{i-1}$ hold for some $1 \leq j < i \leq r+1$. Then, we have 
$f_{\lambda}=X_0^{n-n_j}f_{\lambda_{mP_j}}f_{\lambda'}$, 
where $\lambda'=(0, c-mP_j, \w_{(0, c-mP_j)})$. 
In particular, $f_{\lambda} \in \tilde{K}$. 
\end{lemma}
\begin{proof}
Set $\lambda''=(n-n_j,0,0)$. 
Then we have $\lambda''+\lambda_{mP_j}+\lambda'=(n, c, \w_{(0, c-mP_j)}-n_j)$. Also, we see by Lemma \ref{bigger} that 
$n+\w_{(0,c-mP_j)}-n_j
< n+q-p-n_{i-1}-n_j < q-p$. 
Therefore we have $ \w_{(0,c-mP_j)}-n_j=\w_{(n,c)}$, and hence 
$\lambda''+\lambda_{mP_j}+\lambda'=\lambda$. 
Taking  Lemma \ref{1dim} (ii) into account, it follows that $f_{\lambda}=X_0^{n-n_j}f_{\lambda_{mP_j}}f_{\lambda'}$, since $X_0^{n-n_j}=f_{\lambda''}$. The last statement follows from $f_{\lambda'} \in \tilde{K}$. 
\end{proof}
\begin{proof}[Proof of Theorem \ref{idealO2}]
Set $J=\tilde{J}^i_0$, and let
$J^c=J \cap R^c$, 
$J_n=J \cap R_n$, and 
$J^c_n=J \cap R^c_n=J^c \cap J_n$. 
Then we see that $J_{(n,d)}=\bigoplus _{c \equiv d\; (mod\; m)} J^c_n$, and therefore we have 
\[
R_{(n,d)}/J_{(n,d)} \cong \bigoplus_{c \equiv d\; (mod\; m)} R^c_n/J^c_n.
\]
Recall that $R^c_n=\bigoplus_{\w \geq \w_{(n,c)}} R_{(n,c,\w)}$. Since we have $\bigoplus_{\w > \w_{(n,c)}} R_{(n,c,\w)} \subset J$ by Lemma \ref{q-p}, it suffices to prove that 
\begin{equation}
\dim \left(\bigoplus_{c \equiv d\; (mod\; m)} R_{(n,c,\w_{(n,c)})}/(R_{(n,c,\w_{(n,c)})} \cap J)\right) \leq 1 \label{less}
\end{equation}
holds for any weight $(n,d) \in \mathbb Z \times \mathbb Z/m \mathbb Z$. 
We divide the proof into two steps. 

\emph{Step 1.} We show that \eqref{less} 
holds if $0 \leq n < q-p$ and $d=0$. 
Let  $\lambda=(n,c,\w_{(n,c)}) \in \Lambda_{(n,0)}$. Note that we have $c \geq 0$ by Example \ref{positive}, and recall that every 
$R_{(n,c,\w_{(n,c)})}$ is $1$-dimensional, namely
$R_{(n,c,\w_{(n,c)})}=\langle f_{(n,c,\w_{(n,c)})} \rangle$. 

\emph{Case 1 of Step 1.} Let $0 \leq n < n_{i-1}$. By Lemma \ref{contain}, we see that $f_{\lambda} \in J$ if $c >0$. 
This implies \eqref{less}. 

\emph{Case 2 of Step 1.} Let $n_{i-1} \leq n < q-p$. 
By Lemma \ref{ni}, there is a unique integer $1 \leq j_1 \leq i-1$ such that $n_{j_1} \leq n < n_{j_1-1}$. 
If $n-n_{j_i} \geq  n_{i-1}$, then we can take $1 \leq j_2 \leq i-1$ uniquely to satisfy $n_{j_2} \leq n-n_{j_1} < n_{j_2-1}$. 
By continuing in this way, we get 
$n-(n_{j_1}+n_{j_2}+\dots+n_{j_{u_n-1}}+n_{j_{u_n}}) < n_{i-1}$ for some $1 \leq j_1,\; j_2,\; \dots,\; j_{u_n} \leq i-1$. Namely, we have 
\begin{align*}
n_{j_1} \leq n < n_{j_1-1}, \qquad 
 & n-n_{j_1} \geq n_{i-1},\\
 n_{j_2} \leq n-n_{j_1} < n_{j_2-1}, \qquad 
 & n-(n_{j_1}+n_{j_2}) \geq n_{i-1}, \\
\dots \qquad & \dots\\
 n_{j_{u_n-1}} \leq n-(n_{j_1}+ \dots+n_{j_{u_n-2}}) < n_{j_{u_n-1}-1}, \qquad 
 & n-(n_{j_1}+\dots+n_{j_{u_n-2}}+n_{j_{u_n-1}}) \geq n_{i-1},\\
 n_{j_{u_n}} \leq n-(n_{j_1}+ \dots+n_{j_{u_n-1}}) < n_{j_{u_n}-1}, \qquad 
& n-(n_{j_1}+\dots+n_{j_{u_n-1}}+n_{j_{u_n}}) < n_{i-1}.
\end{align*} 
In the following, we show \eqref{less} by induction on $u_n$. 
Set $u=u_n$ and $P=P_{j_1}+ \dots +P_{j_u}$. 
First suppose that $u=1$. 
Since $j_1 <i$, we have $P < P_i$. 
We show that $f_{\lambda} \in J$ holds if $c \neq mP$. 
If $c=0$, then we have $f_{\lambda}=X_0^n$ by 
Example \ref{positive}, and therefore $f_{\lambda} \in J$. 
If $0 < c < mP$, then we have $f_{\lambda} \in \tilde{K}$ by applying Lemma \ref{contain} (i) with $i=j_1$. 
If $mP < c < mP_i$, then by applying Lemma  \ref{cor of bigger} with $j=j_1$ we see that $f_{\lambda} \in \tilde{K}$. 
If  $c \geq mP_i$, then we have $f_{\lambda} \in (X_1^{e_i}X_3^{l_i})+\tilde{K}$ by Lemma \ref{contain} (iii), (iv).  
Next suppose that $u >1$. 
If $c=0$, then  $f_{\lambda}\in (X_0^{n_{i-1}})$. If $0 <  c< mP_{j_1}$, then  we have $f_{\lambda} \in \tilde{K}$ as above. 
Suppose now that $c>mP_{j_1}$, and set $P'=P-P_{j_1}$, $n'=n-n_{j_1}$, $c'=c-mP_{j_1}$, and $\lambda'=(n', c', \w_{(n', c')})$.  
Since we have $\w_{(n_{j_1},mP_{j_1})}+\w_{(n', c')}+n_{j_1}+n'=\w_{(n', c')}+n' < q-p$ by Example \ref{zi}, it follows from Corollary \ref{remainder1} that $\w_{(n,c)}=\w_{(n_{j_1}, mP_{j_1})}+\w_{(n',c')}$. 
Thus we get 
$\lambda=\lambda_{mP_{j_1}}+\lambda'$, and hence 
$f_{\lambda}=f_{\lambda_{mP_{j_1}}}f_{\lambda'}$  
by Lemma \ref{1dim}. 
Now, since we have $u_{n'}=u-1$, 
it follows from the induction hypothesis and the relation $f_{\lambda}=f_{\lambda_{mP_{j_1}}}f_{\lambda'}$ that 
\eqref{less} holds. 

\emph{Step 2.}  
In this step, we prove that \eqref{less} holds for an arbitrary weight $(n, d)$. 
Let 
$\lambda=(n,c,\w_{(n,c)}) \in \Lambda_{(n,d)}$. 
Set $n'=n+\w_{(n, c_{(n,d)})}$, $c'=c-c_{(n,d)}$, and 
$\lambda'=(n', c', \w_{(n', c')}) \in \Lambda_{(n',0)}$. 
Also, let 
\[
\lambda''=\mu\left(0,\frac{qc_{(n,d)}-\w_{(n,c_{(n,d)})}}{q-p}, \frac{pc_{(n,d)}-\w_{(n,c_{(n,d)})}}{q-p}\right)=(n-n', c_{(n,d)}, \w_{(n,c_{(n,d)})}) \in \Lambda_{(n-n', d)}.
\]
Since $n-n'=-\w_{(n,c_{(n,d)})}$, we have 
$\w_{(n', c')}+\w_{(n,c_{(n,d)})}+n'+n-n'=\w_{(n', c')}+n' <q-p$.  
As in Case 2 of Step 1, we see that  
$f_{\lambda}=f_{\lambda'}f_{\lambda''}$. 
On the other hand we have $0 \leq n' < q-p$ by Lemma \ref{min}, and therefore $\dim R_{(n',0)}/J_{(n',0)} \leq 1$ by Step 1.  
This yields that $\dim R_{(n,d)}/J_{(n,d)} \leq 1$, 
since $f_{\lambda'} \in R_{(n',0)}$. 
\end{proof}
\begin{remark}\label{notin}
Let $\lambda=(n,c,\w) \in \Lambda_{(n,0)}$, where $0 \leq n <q-p$. In view of the proof of Theorem \ref{idealO2}, we 
deduce the following. 
\begin{itemize}
\item Suppose that $0 \leq n \leq n_{i-1}$. Then we have $f_{\lambda} \in \tilde{J}^i_0$ if $\lambda \neq (n,0,\w_{(n,0)})$. 
\item Suppose that $n_{i-1} \leq n < q-p$. Then we have 
$f_{\lambda} \in \tilde{J}^i_0$ if $\lambda \neq (n,mP, \w_{(n,mP)})$. 
\end{itemize}
\end{remark}
\begin{proof}[Proof of Theorem \ref{idealO}]
Set $J=\tilde{J}^i_1$. 
As in the proof of Theorem \ref{idealO2}, we show that 
\begin{equation}
\dim \left(\bigoplus_{c \equiv d\; (mod\; m)} R_{(n,c,\w_{(n,c)})}/(R_{(n,c,\w_{(n,c)})} \cap J)\right) \leq 1 \label{less2}
\end{equation}
holds for any weight $(n,d) \in \mathbb Z \times \mathbb Z/m \mathbb Z$. 

\emph{Step 1.} 
In this step, we show that \eqref{less2}
holds if  $0 \leq n < q-p$ and $d=0$. 
Let $\lambda=(n,c,\w_{(n,c)}) \in \Lambda_{(n,0)}$.  

\emph{Case 1 of Step 1.} Let $0 \leq n < n_i$. If $c >0$, 
then we have $f_{\lambda} \in J$ by  Lemma \ref{contain} (i), (ii), and (v). 


\emph{Case 2 of Step 1.} Let $n_i \leq n < n_{i-1}$. 
Taking 
\[
f_{(n,mP_i, \w_{(n,mP_i)})}-f_{(n,0,\w_{(n,0)})}=X_0^{n-n_i}X_1^{e_i}X_3^{l_i}-X_0^n=X_0^{n-n_i}(X_1^{e_i}X_3^{l_i}-X_0^{n_i}) \in J
\] 
into account, it follows from Lemma \ref{contain} (vi) that  \eqref{less2} holds. 

\emph{Case 3 of Step 1.} Let $n_{i-1} \leq n < q-p$. 
As in Case 2 of the proof of Theorem \ref{idealO2}, we have $n-(n_{j_1}+n_{j_2}+\dots+n_{j_{u_n-1}}+n_{j_{u_n}}) < n_{i-1}$ for some $1 \leq j_1,\; j_2,\; \dots,\; j_{u_n} \leq i-1$. Set $u=u_n$, and $P=P_{j_1}+ \dots +P_{j_u}$. 
Suppose that $u=1$. 
If $0 \leq c< mP$, then we see that  $f_{\lambda} \in J$ in a similar way.  
Let $c \geq mP$. Then we can write  
$f_{\lambda}=f_{\lambda'}f_{\lambda_{mP}}$, where $\lambda'= (n-n_{j_1}, c-mP, \w_{(n-n_{j_1},c-mP)})$.  
If $0 \leq n-n_{j_1} <n_i$, then we see that 
\eqref{less2} holds by applying Lemma \ref{contain} (i), (ii), and (v) for $f_{\lambda'}$. 
If $n_i  \leq n-n_{j_1} <n_{i-1}$, then \eqref{less2} follows from a similar argument to the one we used in Case 2. 

\emph{Step 2.} By arguing as in Step 2 of the proof of Theorem \ref{idealO2}, we deduce that $\dim R_{(n,d)}/J_{(n,d)} \leq 1$ holds for any $(n,d) \in \mathbb Z \times \mathbb Z/m \mathbb Z$. 
\end{proof}
\begin{corollary}\label{D}
The quotient ring $A/J^{r+1}_1$ has Hilbert function $h$. 
\end{corollary}
\begin{proof}
We can easily see that $\mathfrak D$ is the  $SL(2) \times \C^*$-orbit of $\pi(x')$, where $ x'=(0,1,0,1,0) \in H_{q-p}$. 
Let $[J] \in \gamma^{-1}(\mathfrak D)$ be such that $\gamma([J])=\pi(x')$.  
Since $(e_{r+1}, l_{r+1})=(aq, ap) \in M^+_{l,m}$, we have $X_1^{e_{r+1}}X_3^{l_{r+1}} \in R^{\G}$ by Remark \ref{invariant ring}. 
Then by a similar argument as in the proof of \cite[Lemma 4.5]{Ku}, we see that $(X_0^{q-p},\; X_2,\; X_4,\; 1-X_1^{e_{r+1}}X_3^{l_{r+1}}) \subset J$. 
Also, concerning $X_2, X_4 \in J$, it follows from Theorem \ref{generator2} that   $s_1X_0^k+s_2X_1^{e_r}X_3^{l_r} \in J$ holds for some $(s_1, s_2) \neq 0$. 
Since we have $e_{r+1} \geq e_r$ and $l_{r+1} \geq l_r$, the condition $1-X_1^{e_{r+1}}X_3^{l_{r+1}} \in J$ implies that $s_2=0$. 
Therefore, we get $J^{r+1}_1 \subset J$, and hence $\dim (A/J^{r+1}_1)_{(n,d)} \geq \dim (A/J)_{(n,d)}=h(n,d)=1$. Taking Theorem \ref{idealO} into account, we obtain $\dim (A/J^{r+1}_1)_{(n,d)}=h(n,d)$. 
\end{proof}
\begin{remark}
By the proof of Corollary \ref{D}, we have $\gamma^{-1}(\pi(x'))=\{[J^{r+1}_1]\}$. 
We will see  in Corollary \ref{realhilb} that $J^i_1$ and $J^i_0$ have Hilbert function $h$ for any $1 \leq i \leq r+1$. 
\end{remark}
\section{Proof of Main Theorem}\label{final}
We have $\Psi(\mathcal H^{main}) \cong \widetilde{\E'}$ by Proposition \ref{image}. 
Therefore, in order to complete the proof of Theorem \ref{main theorem}, we are left to show that $\Psi |_{\mathcal H^{main
}}$ is injective. 
Indeed, considering the fact that $\widetilde{\E'}$ is normal, 
it follows from the Zariski's Main Theorem that $\Psi |_{\mathcal H^{main}}$ being injective implies that $\Psi|_{\mathcal H^{main}}$ being a closed immersion. 

The weighted blow-up $\E'\cong \varphi(\E') \subset \E \times \mathbb P^1 \times \mathbb P^1$ contains the following four $SL(2) \times \C^*$-orbits: 
\begin{align*}
& \mathfrak U \cong (SL(2) \times \C^*) \cdot (\pi(x), [1:0], [0:1]),\; \mbox{where}\; x=(1,1,0,0,1) \in H_{q-p}; \\
& \mathfrak D \cong (SL(2) \times \C^*) \cdot (\pi(x'), [1:0], [1:0]),\; \mbox{where}\; x'=(0,1,0,1,0) \in H_{q-p}; \\
& C \cong (SL(2) \times \C^*) \cdot (O, [1:0], [1:0]); \\
& C' \cong (SL(2) \times \C^*) \cdot (O, [1:0],[0:1]).
\end{align*}
\begin{lemma}\label{outside}
$\psi |_{\mathcal H^{main}} : \mathcal H^{main} \to \E'$  is bijective outside the closed orbit $C$. 
\end{lemma}
\begin{proof}
We show the bijectivity orbit-wise. 
Taking the construction of $\psi$ and the proof of Corollary \ref{D} into account, we see that 
$\psi([J^{r+1}_1])=(\pi(x'), [1:0], [1:0])$, and that 
\[
\psi^{-1}(\psi([J^{r+1}_1]))=\{[J] \in \mathcal H\; :\; X_2, \; X_4 \in J,\; \gamma([J])=\pi(x')\}=\{[J^{r+1}_1]\}.
\] 
Thus, $\psi |_{\mathcal H^{main}}$ is bijective over $\mathfrak D$. Analogously, we see that 
\begin{align*}
\psi^{-1}(O, [1:0],[0:1])& =\{[I] \in \mathcal H\; :\; X_2,\; X_3 \in I,\; \gamma([I])=O\}\\
&=\{[I] \in \mathcal H\; :\; I_0 \subset I\}=\{[I_0]\}
\end{align*}
concerning Theorem \ref{idealU}. 
Therefore, $\psi |_{\mathcal H^{main}}$ is bijective over $C'$. 
\end{proof}
Recall that the toroidal spherical variety $\widetilde{\E'}$ corresponds to the colored fan $\mathfrak F(\widetilde{\E'})$ having $(\mathcal C_i, \phi)$ $(1 \leq i \leq r+1)$ as its maximal colored cones. 
By Remark \ref{orbit}, each colored cone in $\mathfrak F(\widetilde{\E'})$ corresponds bijectively to an $SL(2) \times \C^*$-orbit in $\widetilde{\E'}$. 
The closed orbit $Y_i$ ($1 \leq i \leq r+1$) corresponds to the maximal colored cone $(\mathcal C_i, \phi)$. 
For each $1 \leq i \leq r$, we denote by $O_i$ the $SL(2) \times \C^*$-orbit corresponds to the colored cone $(\mathbb Q_{\geq 0} \rho_i, \phi)$. 
Let 
\[
g=
\begin{pmatrix}
0 & -1 \\
1 & 0
\end{pmatrix}
\in SL(2),
\] 
and denote by $y_i$ (resp. $y_i'$) the point of $\P(V^{\vee})$ whose coordinates are all $0$ except for the one(s)  corresponding to the basis $g \cdot f_{i-1}^{\vee}$ (resp. the bases $g \cdot f_{i-1}^{\vee}$ and $g \cdot f_i^{\vee}$). 
Then we have 
\[
\Phi(Y_i)=(SL(2) \times \C^*) \cdot (O, y_i)
\]
and 
\[
\Phi(O_i)=(SL(2) \times \C^*) \cdot (O, y_i').
\]
\begin{proposition}\label{inj}
$\Psi|_{\mathcal H^{main}}$ is injective. 
\end{proposition}
\begin{proof}
Taking Lemma \ref{outside} into account, it suffices to show that each of the set-theoretical fibers of $(O,y_i)$ and $(O, y_i')$ consists of one point. 
We show that $\Psi^{-1}(O,y_i)=\{[J^i_0]\}$ and $\Psi^{-1}(O,y_i')=\{[J^i_1]\}$ hold. Let  $[J] \in \Psi^{-1}(O, y_i)$, and write its image under $\eta_{n_0,0}$ and $\eta_{n_i,0}$ as 
\[
\eta_{n_0,0}([J])=[t^{(0)}_{0,0}: t^{(0)}_{e_0,0}: t^{(0)}_{0,l_0}: t^{(0)}_{e_0,l_0}]
\]
and 
\[
\eta_{n_i,0}([J])=[u^{(i)}: t^{(i)}_{0,0}: \dots : t^{(i)}_{e,l}: \dots  : t^{(i)}_{e_i, l_i}] \quad (1 \leq i \leq r)
\]
as in the proof of 
Proposition \ref{morphism}. 
First, it follows from $\gamma([J])=O$ that $K \subset J$. 
Since we have 
\[
g \cdot f_{i-1}=X^{e_0+ e_1+ \dots +e_{i-1}}Y^{l_0+l_1+ \dots +l_{i-1}}=X_0^{-(n_0+n_1+ \dots +n_{i-1})}X_1^{e_0+e_1+\dots +e_{i-1}}X_3^{l_0+l_1+\dots +l_{i-1}},
\] 
we see that $g \cdot f_{i-1}$ maps to 
\begin{align}
&X_1^{e_0} \otimes X_1^{e_1} \otimes \dots \otimes X_1^{e_{i-1}} \otimes X_3^{l_0} \otimes X_3^{l_1} \otimes \dots \otimes X_3^{l_{i-1}} \otimes X_0^{n_{i}} \otimes \dots \otimes X_0^{n_r} \notag \\
& \qquad \quad  \quad \in A(e_0, e_1, \dots , e_{i-1}) \otimes B(l_0, l_1, \dots , l_{i-1}) \otimes C(n_{i}) \otimes \dots \otimes C(n_r) \notag 
\end{align}
under the isomorphism $V \cong \widetilde{V}$. 
Therefore, by the definition of $y_i$, we have 
\begin{equation}
t^{(0)}_{e_0,l_0}  t^{(1)}_{e_1, l_1} \cdots t^{(i-1)}_{e_{i-1}, l_{i-1}}  u^{(i)}  \cdots  u^{(r)}=s  \label{eq1}
\end{equation}
for some $s \in \C^*$. 
Similarly, by paying attention to  the basis $g \cdot f_i^{\vee}$, we have 
\begin{equation}
t^{(0)}_{e_0,l_0} t^{(1)}_{e_1, l_1}  \cdots t^{(i-1)}_{e_{i-1}, l_{i-1}}  t^{(i)}_{e_i, l_i} u^{(i+1)}  \cdots  u^{(r)}=0.  \label{eq2}
\end{equation}
By \eqref{eq1} and \eqref{eq2} we have $t^{(i)}_{e_i, l_i}=0$, which implies that $X_1^{e_i}X_3^{l_i} \in J$.  
Next notice that the vector 
\[
Z^{e_0}X^{e_1+\dots +e_{i-1}}W^{l_0}Y^{l_1+ \dots +l_{i-1}}=X_0^{-(n_0+n_1+ \dots +n_{i-1})}X_2^{e_0}X_1^{e_1+\dots +e_{i-1}}X_4^{l_0}X_3^{l_1+ \dots +l_{i-1}}
\]
maps to 
\begin{align*}
& X_2^{e_0} \otimes X_4^{l_0} \otimes X_1^{e_1} \otimes X_3^{l_1} \otimes \dots \otimes X_1^{e_{i-1}} \otimes X_3^{l_{i-1}} \otimes X_0^{n_i} \dots \otimes X_0^{n_r} \\
& \quad \in A(e_0) \otimes B(l_0) \otimes A(e_1) \otimes B(l_1) \otimes \dots \otimes A(e_{i-1}) \otimes B(l_{i-1}) \otimes C(n_i) \otimes \dots \otimes C(n_r)
\end{align*}
under $V \cong \widetilde{V} \subset V'$, which yields that 
\begin{equation}
t^{(0)}_{0,0} t^{(1)}_{e_1, l_1} \cdots t^{(i-1)}_{e_{i-1}, l_{i-1}}  u^{(i)} \cdots u^{(r)}=0.\label{eq3}
\end{equation}
Comparing \eqref{eq1} and \eqref{eq3}, we have $t^{(0)}_{0,0}=0$, which implies that $X_2^{e_0}X_4^{l_0}=X_2X_4 \in J$. 
In a similar way, we also have $X_2X_3,\; X_1X_4 \in J$. Concerning \eqref{s1s2} and \eqref{s3s4}, it follows that  $(X_2, X_4) \subset J$ . Therefore, we get $(X_0^{q-p}, X_2, X_4, X_1^{e_i}X_3^{l_i}) +K \subset J$. 
Now, suppose that $i=1$. Then we have $J^1_0 \subset J$, since $n_0=q-p$. Taking Theorem \ref{idealO} into account, it follows that $J^1_0=J$, and thus we get $\Psi^{-1}(O,y_1)=\{[J^1_0]\}$. 
Next, suppose that $i >1$. Since the vector
\[
X^{e_0+e_1+ \dots +e_{i-2}}Y^{l_0+l_1+ \dots +l_{i-2}}=X_0^{-(n_0+n_1+ \dots +n_{i-2})}X_1^{e_0+e_1+ \dots +e_{i-2}}X_3^{l_0+l_1+ \dots +l_{i-2}}
\]
maps to 
\[
X_1^{e_0} \otimes X_3^{l_0} \otimes X_1^{e_1} \otimes X_3^{l_1} \otimes \dots \otimes X_1^{e_{i-2}}X_3^{l_{i-2}} \otimes X_0^{n_{i-1}} \otimes \dots \otimes X_0^{n_r}
\]
under the isomorphism $V \cong \widetilde{V} \subset V'$, we see that 
\begin{equation}
t^{(0)}_{e_0,l_0}  t^{(1)}_{e_1, l_1}  \cdots t^{(i-2)}_{e_{i-2}, l_{i-2}}  u^{(i-1)} \cdots u^{(r)}=0 \label{eq4}.
\end{equation}
By \eqref{eq1} and \eqref{eq4}, we get $u^{(i-1)}=0$, and hence $X_0^{i-1} \in J$. 
Summarizing, we get $J^i_0 \subset J$. Therefore, we have $J^i_0=J$ and $\Psi^{-1}(O,y_i)=\{[J^i_0]\}$. 
Next, let $[I] \in \Psi^{-1}(O, y_i')$, and  write
\[
\eta_{n_0,0}([I])=[t^{(0)}_{0,0}: t^{(0)}_{e_0,0}: t^{(0)}_{0,l_0}: t^{(0)}_{e_0,l_0}]
\]
and 
\[
\eta_{n_i,0}([I])=[u^{(i)}: t^{(i)}_{0,0}: \dots : t^{(i)}_{e,l}: \dots  : t^{(i)}_{e_i, l_i}] \quad (1 \leq i \leq r).
\]
as above. 
In a similar manner, we can show that 
 $(X_0^{n_{i-1}},\; X_2,\; X_4)+K \subset I$. 
Moreover, we see that 
\[
t^{(0)}_{e_0,l_0} t^{(1)}_{e_1, l_1} \cdots t^{(i-1)}_{e_{i-1}, l_{i-1}}  u^{(i)} u^{(i+1)}\cdots u^{(r)}=s
\]
and 
\[
t^{(0)}_{e_0,l_0} t^{(1)}_{e_1, l_1}  \cdots t^{(i-1)}_{e_{i-1}, l_{i-1}}  t^{(i)}_{e_i, l_i}  u^{(i+1)} \cdots  u^{(r)}=s
\] 
hold for some $s \in \C^*$. 
Therefore, we get $u^{(i)}=t^{(i)}_{e_i, l_i}$. Since we have already seen that 
 $X_2, X_4 \in I$, this implies that $\eta_{n_i,0}([I])=[1: 0:\dots :0: 1]$. It follows that $X_0^{n_i}-X_1^{e_i}X_3^{l_i} \in I$ concerning the construction of $\eta_{n_i,0}$. 
 As a consequence, we get $J^i_1 \subset I$, and therefore $I=J^i_1$. 
\end{proof}
\begin{corollary}\label{realhilb}
The quotient rings $A/J^i_1$ and $A/J^i_0$ have  
Hilbert function $h$ for any  $1 \leq i \leq r+1$. 
\end{corollary}
\begin{remark}\label{notin2}
Let $\lambda=(n,c,\w) \in \Lambda_{(n,0)}$, where $0 \leq n <q-p$. 
Taking Remark \ref{notin} and Corollary \ref{realhilb} into account, we see that the following properties are true. 
\begin{itemize}
\item Suppose that $0 \leq n \leq n_{i-1}$. Then we have $f_{\lambda} \in \tilde{J}^i_0$ if and only if $\lambda \neq (n,0,\w_{(n,0)})$. 
\item Suppose that $n_{i-1} \leq n < q-p$. Then we have 
$f_{\lambda} \in \tilde{J}^i_0$ if and only if $\lambda \neq (n,mP, \w_{(n,mP)})$. 
\end{itemize}
\end{remark}
Let us denote by $\mathcal H^{\tilde{B}}$ the set of $\tilde{B}$-fixed points of $\mathcal H$. 
\begin{corollary}\label{Borel-fixed points}
We have $\mathcal H^{\tilde{B}}=\{[J^1_0],\; \dots,\; [J^{r+1}_0]\}$.
\end{corollary}
\begin{proof}
Let $[J] \in \mathcal H^{\tilde{B}}$. 
Then, we have $s_1X_1+s_2X_2 \in J$ for some $(s_1,s_2) \neq 0$ by \eqref{s1s2}. 
Since $J$ is stable under the action of $\tilde{B}$, we have $X_2 \in J$.  
Similarly, we have $X_4 \in J$ by \eqref{s3s4}. 
Therefore, $(X_2, X_4)+K \subset J$ concerning $\gamma([J])=O$. 
By Theorem \ref{generator2}, we see that either 
$X_0^{n_j} \in J$ or $X_1^{e_j}X_3^{l_j} \in J$ holds for any $1 \leq j\leq r+1$, since $h(n_j, 0)=1$. 
Let 
$i=\min\{j\; :\; X_1^{e_j}X_3^{l_j} \in J\}$.
Then, we have $(X_0^{n_{i-1}},\; X_1^{e_i}X_3^{l_i})  \subset J$, and hence $J^i_0 \subset J$. This implies that $J^i_0=J$, since both $J^i_0$ and $J$ have Hilbert function $h$. 
\end{proof}
\begin{corollary}\label{duval}
The invariant Hilbert scheme $\mathcal H$ coincides with $\mathcal H^{main}$. 
\end{corollary}
\begin{proof}
By \cite[Lemma 1.6]{Ter14}, we see that every closed subset of $\mathcal H$ contains at least one fixed point for the action of $\tilde{B}$. 
Therefore it follows that $\mathcal H$ is connected, since Corollary \ref{Borel-fixed points} implies that every $\tilde{B}$-fixed point is contained in $\mathcal H^{main}$. 
In the following, we show that $\mathcal H$ is smooth. 
Concerning \cite[Proposition 3.5]{B} and the proof of \cite[Lemma 1.7]{Ter14}, it suffices to show that 
$\dim \Hom^{\G}_S(J^i_0, A/J^i_0)=\dim \mathcal H^{main}=3$  
holds for any $[J^i_0] \in \mathcal H^{\tilde{B}}$. 
Recall that we have seen in Remark \ref{another} that 
\[
J^i_0=(X_0^{n_{i-1}},\; X_2,\; X_4,\; X_1^{e_i}X_3^{l_i},\; F_1,\; \dots,\; F_{b-1}).  
\]
Let $\phi \in \Hom^{\G}_{S}(J^i_0, A/J^i_0)$. 
Since $\phi$ is $\G$-equivariant, we have 
\begin{align*}
& \phi(X_0^{n_{i-1}}) =\alpha_1 X_1^{e_{i-1}}X_3^{l_{i-1}},\; \phi(X_2)=\alpha_2 X_1,\;  \phi(X_4) =\alpha_3 X_3,\;  
 \phi(X_1^{e_i}X_3^{l_i}) = \alpha_4 X_0^{n_i},\\
 & 
 \phi(F_j) =\beta_j\quad (1 \leq j \leq b-1)
\end{align*}
for some $\alpha_1,\; \alpha_2,\; \alpha_3,\; \alpha_4,\; \beta_j \in \C$. 
Also, since $\phi$ is a homomorphism of $S$-modules, we have
\begin{align*}
0
 =\phi(X_0^{q-p}-X_1X_4+X_2X_3)
&= \alpha_1 X_0^{q-p-n_{i-1}}X_1^{e_{i-1}}X_3^{l_{i-1}} -\alpha_3 X_1X_3+\alpha_2X_1X_3\\
& =
\begin{cases}
(\alpha_1+\alpha_2-\alpha_3)X_1X_3 \qquad (\mbox{if}\; i=1) \\
(\alpha_2-\alpha_3)X_1X_3 \qquad (\mbox{otherwise})
\end{cases}
\end{align*}
concerning that  $X_0^{q-p-n_{i-1}}X_1^{e_{i-1}}X_3^{l_{i-1}}=X_1X_3 F_{P_{i-1}} \in J^i_0$ holds for every $i >1$. 

In the following, we show that $\beta_j=0$ holds for any 
$1 \leq j \leq b-1$. Set 
\[
d_1=\frac{qmj-\w_{(0,mj)}}{q-p}, \quad d_3=\frac{pmj-\w_{(0,mj)}}{q-p}.
\]
Then we have $F_j=X_0^{\w_{(0,mj)}}X_1^{d_1}X_3^{d_3}$ by Lemma \ref{1dim}. 
Also, note that we have $F_{P_i}=X_0^{q-p-n_i}X_1^{e_i-1}X_3^{l_i-1}$, $f_{\lambda_{mP_i}}=X_1^{e_i}X_3^{l_i}$, and  $f_{\lambda_{mj}}=X_1^{d_1+1}X_3^{d_3+1}$.

\emph{Case 1.} First suppose that $j > P_i$.  Then, taking Lemma \ref{cor of rem} into account, we see that $d_1+1 > e_i$ and $d_3+1 > l_i$ hold. Set $f=X_0^{\w_{(0,mj)}}X_1^{d_1+1-e_i}X_3^{d_3+1-l_i}$. 
Then, 
\[
0=\phi(X_1X_3 F_j-X_1^{e_i}X_3^{l_i}f)=\beta_jX_1X_3-\alpha_4X_0^{n_i}f.
\]  
Since $X_0^{n_i}f \in R^{mj-mP_i}_{q-p}$, it follows from Proposition \ref{long} that $X_0^{n_i}f \in K \subset J^i_0$. 
Concerning $X_1X_3 \in R^0_{q-p}$, we see that $X_1X_3 \notin J^i_0$, since otherwise we get  $\dim (A/J^i_0)_{(q-p,0)}=0$, which contradicts to Corollary \ref{realhilb}. Therefore, we have $\beta_j=0$. 

\emph{Case 2.} Next, if $j=P_i$ then we have 
\[
0=\phi(X_1X_3 F_{P_i}-X_0^{q-p-n_i}X_1^{e_i}X_3^{l_i})=
\beta_{P_i}X_1X_3-\alpha_4X_0^{q-p}, 
\]
and hence $\beta_{P_i}=0$, since $X_1X_3 \notin J^i_0$ and $X_0^{q-p} \in J^i_0$. 

\emph{Case 3.} Lastly, we consider the case where $1 \leq j < P_i$. Following the same line as in Case 1, we see that the condition $j <P_i$ implies $d_1 < e_i$ and $d_3 <l_i$.  Set $n=\w_{(0,mj)}+n_i$, and $c=m(P_i-j)$. 
Then we have $n_{i-1} \leq  n \leq q-p-n_{i-1}+n_i<q-p$ by Lemma \ref{bigger}. Also, we see that $X_1^{e_i-d_1}X_3^{l_i-d_3}=f_{\lambda_c}$. 
Therefore, we have 
\[
0=\phi(X_1^{e_i-d_1}X_3^{l_i-d_3} F_j-X_0^{\w_{(0,mj)}}X_1^{e_i}X_3^{l_i})=\beta_j f_{\lambda_c}-\alpha_4X_0^{n} .
\]
It immediately follows from $n_{i-1} \leq n$ that $X_0^{n} \in J^i_0$, and we are left to show that $f_{\lambda_c} \notin J^i_0$. As in Case 2 of the proof of Theorem \ref{idealO2}, we have \[
n-(n_{j_1}+n_{j_2}+\dots+n_{j_{u_n-1}}+n_{j_{u_n}}) < n_{i-1}
\]
for some $1 \leq j_1,\; j_2,\; \dots,\; j_{u_n} \leq i-1$. Set $P=P_{j_1}+\dots+ P_{j_{u_n}}$, and $\lambda=(n,mP, \w_{(n,mP)})$. We show that $f_{\lambda_c}$ coincides with $f_{\lambda}$. Namely, we show that $c=mP$. First, we have $f_{\lambda_c} \notin (X_1^{e_i}X_3^{l_i})$ by Lemma \ref{cor of rem}. Next, we claim that $f_{\lambda_c} \notin 
K$. Indeed, if we have $f_{\lambda_c} \in K$, then $f_{\lambda_c} \in (F_1,\; \dots,\; F_{b-1})$. On the other hand, we see that for any $1 \leq l \leq b-1$ the degree of $F_l$ with respect to $X_0$ is greater than $0$, which contradicts to $f_{\lambda_c} \in (F_1,\; \dots,\; F_{b-1})$. 
Therefore, we have $c=mP$ concerning the proof of Theorem \ref{idealO2}. It follows from Remark \ref{notin2} that $f_{\lambda_c} \notin \tilde{J}^i_0$, and 
thus we get $\beta_j=0$. 

Therefore we obtain 
$\dim \Hom^{\G}_{S}(J^i_0, A/J^i_0)  \leq 3$,  
and hence the equality. 
\end{proof}

\paragraph{Acknowledgement}  
The author would like to express her gratitude to Professor Yasunari Nagai, her supervisor, for his valuable discussion and continued unwavering encouragement. 
She is grateful to Professor Hajime Kaji for his beneficial  advice and kind support. 
She would also like to thank Professors Daizo Ishikawa, Ryo Ohkawa, and Taku Suzuki for their useful comments and helpful suggestions.


\newpage

\begin{bibdiv}
\begin{biblist}
\bib{AB}{article}{
   author={Alexeev, Valery},
   author={Brion, Michel},
   title={Moduli of affine schemes with reductive group action},
   journal={J. Algebraic Geom.},
   volume={14},
   date={2005},
   number={1},
   pages={83--117},
 }
\bib{IUJA}{book}{
   author={Arzhantsev, Ivan},
   author={Derenthal, Ulrich},
   author={Hausen, J\"urgen},
   author={Laface, Antonio},
   title={Cox rings},
   series={Cambridge Studies in Advanced Mathematics},
   volume={144},
   publisher={Cambridge University Press, Cambridge},
   date={2015},
   pages={viii+530},
   }
\bib{BH}{article}{
   author={Batyrev, Victor},
   author={Haddad, Fatima},
   title={On the geometry of ${\rm SL}(2)$-equivariant flips},
   language={English, with English and Russian summaries},
   journal={Mosc. Math. J.},
   volume={8},
   date={2008},
   number={4},
   pages={621--646, 846},
   }
\bib{Bec}{article}{
   author={Becker, Tanja},
   title={An example of an ${\rm SL}_2$-Hilbert scheme with
   multiplicities},
   journal={Transform. Groups},
   volume={16},
   date={2011},
   number={4},
   pages={915--938},
  }
\bib{B}{article}{
   author={Brion, Michel},
   title={Invariant Hilbert schemes},
   conference={
      title={Handbook of moduli. Vol. I},
   },
   book={
      series={Adv. Lect. Math. (ALM)},
      volume={24},
      publisher={Int. Press, Somerville, MA},
   },
   date={2013},
   pages={64--117},
   }
\bib{BP}{article}{
   author={Brion, Michel},
   author={Pauer, Franz},
   title={Valuations des espaces homog\`enes sph\'eriques},
   language={French},
   journal={Comment. Math. Helv.},
   volume={62},
   date={1987},
   number={2},
   pages={265--285},
  }
\bib{Bud}{article}{
   author={Budmiger, Jonas},
   title={Deformation of Orbits in Minimal Sheets},
   journal={Dissertation, Universit$\Ddot{a}$t Basel},
   volume={16},
   date={2010},
   number={4},
   pages={915--938},
   }
\bib{CLS}{book}{
   author={Cox, David A.},
   author={Little, John B.},
   author={Schenck, Henry K.},
   title={Toric varieties},
   series={Graduate Studies in Mathematics},
   volume={124},
   publisher={American Mathematical Society, Providence, RI},
   date={2011},
   pages={xxiv+841},
 }
 \bib{F}{book}{
   author={Fulton, William},
   title={Introduction to toric varieties},
   series={Annals of Mathematics Studies},
   volume={131},
   note={The William H. Roever Lectures in Geometry},
   publisher={Princeton University Press, Princeton, NJ},
   date={1993},
   pages={xii+157},
}
\bib{G}{article}{
   author={Ga{\u\i}fullin, S. A.},
   title={Affine toric ${\rm SL}(2)$-embeddings},
   language={Russian, with Russian summary},
   journal={Mat. Sb.},
   volume={199},
   date={2008},
   number={3},
   pages={3--24},
      translation={
      journal={Sb. Math.},
      volume={199},
      date={2008},
      number={3-4},
      pages={319--339},
   },
}

\bib{Knop}{article}{
   author={Knop, Friedrich},
   title={The Luna-Vust theory of spherical embeddings},
   conference={
      title={Proceedings of the Hyderabad Conference on Algebraic Groups},
      address={Hyderabad},
      date={1989},
   },
   book={
      publisher={Manoj Prakashan, Madras},
   },
   date={1991},
   pages={225--249},
}
\bib{K}{book}{
   author={Kraft, Hanspeter},
   title={Geometrische Methoden in der Invariantentheorie},
   language={German},
   series={Aspects of Mathematics, D1},
   publisher={Friedr. Vieweg \& Sohn, Braunschweig},
   date={1984},
   pages={x+308},
  }
\bib{Ku}{article}{
author={Kubota, Ayako},
title={Invariant Hilbert scheme resolution of Popov's $SL(2)$-varieties I: the toric case},
note={preprint},
date={2018},
} 
\bib{Ter}{article}{
   author={Lehn, Christian},
   author={Terpereau, Ronan},
   title={Invariant deformation theory of affine schemes with reductive
   group action},
   journal={J. Pure Appl. Algebra},
   volume={219},
   date={2015},
   number={9},
   pages={4168--4202},
   }
\bib{LV}{article}{
   author={Luna, D.},
   author={Vust, Th.},
   title={Plongements d'espaces homog\`enes},
   language={French},
   journal={Comment. Math. Helv.},
   volume={58},
   date={1983},
   number={2},
   pages={186--245},
}
\bib{Pa}{article}{
   author={Panyushev, D. I.},
   title={Resolution of singularities of affine normal quasihomogeneous
   ${\rm SL}_2$-varieties},
   language={Russian},
   journal={Funktsional. Anal. i Prilozhen.},
   volume={22},
   date={1988},
   number={4},
   pages={94--95},
      translation={
      journal={Funct. Anal. Appl.},
      volume={22},
      date={1988},
      number={4},
      pages={338--339 (1989)},
      },
   }
\bib{Pan}{article}{
   author={Panyushev, D. I.},
   title={The canonical module of an affine normal quasihomogeneous ${\rm
   SL}_2$-variety},
   language={Russian},
   journal={Mat. Sb.},
   volume={182},
   date={1991},
   number={8},
   pages={1211--1221},
   translation={
      journal={Math. USSR-Sb.},
      volume={73},
      date={1992},
      number={2},
      pages={569--578},
   },
}
\bib{Perrin}{article}{
   author={Perrin, Nicolas},
   title={On the geometry of spherical varieties},
   journal={Transform. Groups},
   volume={19},
   date={2014},
   number={1},
   pages={171--223},
  }
\bib{P}{article}{
   author={Popov, V. L.},
   title={Quasihomogeneous affine algebraic varieties of the group ${\rm
   SL}(2)$},
   language={Russian},
   journal={Izv. Akad. Nauk SSSR Ser. Mat},
   volume={37},
   date={1973},
   pages={792--832},
   }
\bib{Ter14}{article}{
   author={Terpereau, R.},
   title={Invariant Hilbert schemes and desingularizations of quotients by
   classical groups},
   journal={Transform. Groups},
   volume={19},
   date={2014},
   number={1},
   pages={247--281},
   }

\bib{Tim}{book}{
   author={Timashev, Dmitry A.},
   title={Homogeneous spaces and equivariant embeddings},
   series={Encyclopaedia of Mathematical Sciences},
   volume={138},
   note={Invariant Theory and Algebraic Transformation Groups, 8},
   publisher={Springer, Heidelberg},
   date={2011},
   pages={xxii+253},
  }

\end{biblist}
\end{bibdiv}
\end{document}